\documentclass[10pt]{amsart}

\usepackage{amsmath, amscd, amssymb,xypic, euler}
\usepackage[frame,cmtip,arrow,matrix,line,graph,curve]{xy}
\usepackage{graphpap, color, pstricks}
\usepackage[mathscr]{eucal}

\newtheorem{theorem}{Theorem}[section]
\newtheorem{proposition}[theorem]{Proposition}
\newtheorem{corollary}[theorem]{Corollary}
\newtheorem{lemma}[theorem]{Lemma}

\theoremstyle{definition}
\newtheorem{definition}[theorem]{Definition}

\newtheorem{remark}[theorem]{Remark}

\newcommand{\PP}{\mathbb{P}}
\newcommand{\QQ}{\mathbb{Q}}
\newcommand{\CC}{\mathbb{C}}

\newcommand{\ZZ}{\mathbb{Z}}

\newcommand{\cO}{\mathcal{O} }
\newcommand{\cA}{\mathcal{A} }
\newcommand{\cI}{\mathcal{I}}

\newcommand{\proj}{\mathrm{Proj}\;}

\def\mapright#1{\,\smash{\mathop{\longrightarrow}\limits^{#1}}\,}

\def\Mzn{\overline{M}_{0,n} }
\def\Mzez{\overline{M}_{0,n\cdot \epsilon_0} }
\def\Mzeo{\overline{M}_{0,n\cdot \epsilon_1} }
\def\Mzet{\overline{M}_{0,n\cdot \epsilon_2} }
\def\Mzemt{\overline{M}_{0,n\cdot \epsilon_{m-2}} }
\def\Mzemo{\overline{M}_{0,n\cdot \epsilon_{m-1}} }
\def\Mzemth{\overline{M}_{0,n\cdot \epsilon_{m-3}} }
\def\Mze{\overline{M}_{0,n\cdot \epsilon} }
\def\Mzek{\overline{M}_{0,n\cdot \epsilon_k} }
\def\Mzeko{\overline{M}_{0,n\cdot \epsilon_{k-1}} }
\def\git{/\!/ }

\begin{document}
\title[Moduli spaces of weighted pointed
stable rational curves]{Moduli spaces of weighted pointed stable
rational curves via GIT}
\date{March 2010}

\author{Young-Hoon Kiem}
\address{Department of Mathematics and Research Institute
of Mathematics, Seoul National University, Seoul 151-747, Korea}
\email{kiem@math.snu.ac.kr}

\author{Han-Bom Moon}
\address{Department of Mathematics, Seoul National University, Seoul 151-747, Korea}
\email{spring-1@snu.ac.kr}

\thanks{Partially supported by an NRF grant}

\begin{abstract}
We construct the Mumford-Knudsen space $\Mzn$ of $n$-pointed
stable rational curves by a sequence of explicit blow-ups from the
GIT quotient $(\PP^1)^n\git SL(2)$ with respect to the symmetric
linearization $\cO(1,\cdots,1)$. The intermediate blown-up spaces
turn out to be the moduli spaces of weighted pointed stable curves
$\Mze$ for suitable ranges of $\epsilon$. As an application, we
provide a new unconditional proof of M. Simpson's Theorem about
the log canonical models of $\Mzn$. We also give a basis of the
Picard group of $\Mze$.
\end{abstract}

\maketitle

\section{Introduction}\label{sec1}
Recently there has been a tremendous amount of interest in the
birational geometry of moduli spaces of stable curves. See for
instance \cite{AlexSwin, FedoSmyt,Hassett,Kapranov,Keel,Li,
Mustata, Simpson} for the genus 0 case only. Most prominently, it
has been proved in \cite{AlexSwin, FedoSmyt,Simpson} that the log
canonical models for $(\Mzn, K_{\Mzn}+\alpha D)$, where $D$ is the
boundary divisor and $\alpha$ is a rational number, give us
Hassett's moduli spaces $\Mze$ of weighted pointed stable curves
with \emph{symmetric} weights $n\cdot \epsilon =(\epsilon,\cdots,
\epsilon)$. See \S\ref{sec2.1} for the definition of $\Mze$ and
Theorem \ref{thm1.2} below for a precise statement. The purpose of
this paper is to prove that actually all the moduli spaces $\Mze$
can be constructed by explicit blow-ups from the GIT quotient
$(\PP^1)^n\git SL(2)$ with respect to the symmetric linearization
$\cO(1,\cdots,1)$ where $SL(2)$ acts on $(\PP^1)^n$ diagonally.
More precisely, we prove the following.
\begin{theorem}\label{thm1.1} There is a sequence of blow-ups
\begin{equation}\label{eq1}
\Mzn=\Mzemt\to \Mzemth\to \cdots \to \Mzet\to \Mzeo\to
(\PP^1)^n\git SL(2) \end{equation} where
$m=\lfloor\frac{n}{2}\rfloor$ and $\frac1{m+1-k}<\epsilon_k\le
\frac1{m-k}$. Except for the last arrow when $n$ is even, the
center for each blow-up is a union of transversal smooth
subvarieties of same dimension.  When $n$ is even, the last arrow
is the blow-up along the singular locus which consists of
$\frac12\binom{n}{m}$ points in $(\PP^1)^n\git SL(2)$, i.e.
$\Mzeo$ is Kirwan's partial desingularization (see \cite{Kirwan})
of the GIT quotient $(\PP^1)^{2m}\git SL(2)$.\end{theorem} If the
center of a blow-up is the transversal union of smooth
subvarieties in a nonsingular variety, the result of the blow-up
is isomorphic to that of the sequence of smooth blow-ups along the
irreducible components of the center in any order (see
\S\ref{sec2.3}). So each of the above arrows can be decomposed
into the composition of smooth blow-ups along the irreducible
components.

As an application of Theorem \ref{thm1.1}, we give a new proof of
the following theorem of M. Simpson (\cite{Simpson}) without
relying on Fulton's conjecture.
\begin{theorem}\label{thm1.2}
Let $\alpha$ be a rational number satisfying $\frac{2}{n-1} < \alpha
\le 1$ and let $D=\Mzn-M_{0,n}$ denote the boundary divisor. Then
the log canonical model $$\Mzn(\alpha) = \proj\left(\bigoplus_{l \ge
0} H^0(\Mzn, \cO(\lfloor l (K_{\Mzn} +\alpha D) \rfloor))\right)
$$
satisfies the following:
\begin{enumerate} \item If $\frac{2}{m-k+2} < \alpha \le \frac{2}{m-k+1}$ for $1\le k\le
m-2$, then $\Mzn(\alpha) \cong \Mzek$.\item  If $\frac{2}{n-1} <
\alpha \le \frac{2}{m+1}$, then $\Mzn(\alpha) \cong (\PP^1)^n \git
G$ where the quotient is taken with respect to the symmetric
linearization $\cO(1,\cdots,1)$.
\end{enumerate}
\end{theorem}
There are already two different \emph{unconditional} proofs of
Theorem \ref{thm1.2} by Alexeev-Swinarski \cite{AlexSwin} and by
Fedorchuk-Smyth \cite{FedoSmyt}. See Remark \ref{rem-compproofs}
for a brief outline of the two proofs. In this paper we obtain the
ampleness of some crucial divisors directly from Theorem
\ref{thm1.1}. As another application, we give an explicit basis of
the Picard group of $\Mzek$ for each $k$.

\bigskip

It is often the case in moduli theory that adding an extra structure
makes a problem easier. Let $0\le k< n$. A pointed nodal curve
$(C,p_1,\cdots, p_n)$ of genus $0$ together with a morphism $f:C\to
\PP^1$ of degree $1$ is called \emph{$k$-stable} if
\begin{enumerate}
\item[i.] all marked points $p_i$ are smooth points of $C$;
\item[ii.] no more than $n-k$ of the marked points $p_i$ can coincide;
\item[iii.] any ending irreducible component $C'$ of $C$ which is
contracted by $f$ contains more than $n-k$ marked points;
\item[iv.] the group of automorphisms of $C$ preserving $f$ and $p_i$
is finite.
\end{enumerate}
A. Mustata and M. Mustata prove the following in \cite{Mustata}.
\begin{theorem} \cite[\S1]{Mustata}
There is a fine moduli space $F_k$ of $k$-stable
pointed parameterized curves $(C,p_1,\cdots,p_n,f)$.
Furthermore, the moduli spaces $F_k$ fit into a
sequence of blow-ups
\begin{equation}\label{2010eq1}
\xymatrix{ \PP^1[n]\ar@{=}[r] & F_{n-2}\ar[r]^{\psi_{n-2}} &
F_{n-3}\ar[r]^{\psi_{n-3}} & \cdots \ar[r]^{\psi_2} &
F_1\ar[r]^{\psi_1} & F_0\ar@{=}[r] & (\PP^1)^n }
\end{equation}
whose centers are transversal unions of smooth subvarieties.
\end{theorem}
The first term $\PP^1[n]$ is the Fulton-MacPherson
compactification of the configuration space of $n$ points in
$\PP^1$ constructed in \cite{FM}. The blow-up centers are
transversal unions of smooth subvarieties and hence we can further
decompose each arrow into the composition of smooth blow-ups along
the irreducible components in any order. This blow-up sequence is
actually a special case of L. Li's inductive construction of
a \emph{wonderful compactification} of the configuration space and
transversality of various subvarieties is a corollary of Li's
result \cite[Proposition 2.8]{Li}. (See \S\ref{sec2.3}.) The
images of the blow-up centers are invariant under the diagonal
action of $SL(2)$ on $(\PP^1)^n$ and so this action lifts to $F_k$
for all $k$. The aim of this paper is to show that the GIT
quotient of the sequence \eqref{2010eq1} by $SL(2)$ gives us
\eqref{eq1}.

To make sense of GIT quotients, we need to specify a linearization
of the action of $G=SL(2)$ on $F_k$. For $F_0=(\PP^1)^n$, we
choose the symmetric linearization $L_0=\cO(1,\cdots, 1)$.
Inductively, we choose $L_k=\psi^*_kL_{k-1}\otimes
\cO(-\delta_kE_{k})$ where $E_k$ is the exceptional divisor of
$\psi_k$ and $0<\delta_k<\!<\delta_{k-1}<\!< \cdots
<\!<\delta_1<\!<1$. Let $F_k^{ss}$ (resp. $F_k^s$) be the
semistable (resp. stable) part of $F_k$ with respect to $L_k$.
Then by \cite[\S3]{Kirwan} or \cite[Theorem 3.11]{Hu}, we have
\begin{equation}\label{2010eq2}
\psi_k^{-1}(F^s_{k-1})\subset F_k^s\subset F_k^{ss}\subset
\psi_k^{-1}(F_{k-1}^{ss}).
\end{equation}
In particular, we obtain a sequence of morphisms
$$\bar\psi_k:F_k\git G\to F_{k-1}\git G.$$ It is well known that a
point $(x_1,\cdots,x_n)$ in $F_0=(\PP^1)^n$ is stable (resp.
semistable) if $\ge \lfloor \frac{n}2\rfloor$ points (resp. $>
\lfloor \frac{n}2\rfloor$ points) do not coincide (\cite{MFK,K1}).

Let us first consider the case where $n$ is odd. In this case,
$F_0^s=F_0^{ss}$ because $\frac{n}2$ is not an integer. Hence
$F_k^s=F_k^{ss}$ for any $k$ by \eqref{2010eq2}. Since the blow-up
centers of $\psi_k$ for $k\le m+1$ lie in the unstable part, we
have $F_k^s=F_0^s$ for $k\le m+1$. Furthermore, the stabilizer
group of every point in $F_k^s$ is $\{\pm 1\}$, i.e. $\bar
G=PGL(2)$ acts freely on $F_k^s$ for $0\le k\le n-2$ and thus
$F_k\git G=F_k^s/G$ is nonsingular. By the stability conditions,
forgetting the degree 1 morphism $f:C\to \PP^1$ gives us an
invariant morphism $F_{n-m+k}^s\to \Mzek$ which induces a morphism
$$\phi_k:F_{n-m+k}\git G\to \Mzek \quad \text{for } k=0,\cdots, m-2.$$
Since both varieties are nonsingular, we can conclude that
$\phi_k$ is an isomorphism by showing that the Picard numbers are
identical. Since $\bar G$ acts freely on $F_{n-m+k}^s$, the
quotient of the blow-up center of $\psi_{n-m+k+1}$ is again a
transversal union of $\binom{n}{m-k}$ smooth varieties
$\Sigma^S_{n-m+k}\git G$ for a subset $S$ of $\{1,\cdots, n\}$
with $|S|=m-k$, which are isomorphic to the moduli space
$\overline{M}_{0,(1,\epsilon_k,\cdots,\epsilon_k)}$ of weighted
pointed stable curves with $n-m+k+1$ marked points (Remark
\ref{2010rem1}). Finally we conclude that
$$\varphi_k:\Mzek\cong F_{n-m+k}\git G\mapright{\bar\psi_{n-m+k}}
F_{n-m+k-1}\git G\cong \Mzeko$$ is a blow-up by using a lemma in
\cite{Kirwan} which tells us that quotient and blow-up commute.
(See \S\ref{sec2.2}.) It is straightforward to check that this
morphism $\varphi_k$ is identical to Hassett's natural morphisms
(\S\ref{sec2.1}). Note that the isomorphism
$$\phi_{m-2}:\PP^1[n]\git G\mapright{\cong} \Mzn$$ was obtained
by Hu and Keel (\cite{HuKeel}) when $n$ is odd because $L_0$ is a
\emph{typical} linearization in the sense that $F_0^{ss}=F_0^s$.
The above proof of the fact that $\phi_k$ is an isomorphism in the
odd $n$ case is essentially the same as Hu-Keel's. However their
method does not apply to the even degree case.

\def\Fk{F_{n-m+k} }
\def\tFk{\tilde{F}_{n-m+k} }
\def\Yk{Y_{n-m+k} }

The case where $n$ is even is more complicated because
$F_k^{ss}\ne F_k^s$ for all $k$. Indeed, $F_m\git G=\cdots
=F_0\git G= (\PP^1)^n\git G$ is singular with exactly
$\frac12\binom{n}{m}$ singular points. But for $k\ge 1$, the GIT
quotient of $F_{n-m+k}$ by $G$ is nonsingular and we can use
Kirwan's partial desingularization of the GIT quotient
$F_{n-m+k}\git G$ (\cite{Kirwan}). For $k\ge 1$, the locus
$Y_{n-m+k}$ of closed orbits in $F_{n-m+k}^{ss}-F_{n-m+k}^s$ is
the disjoint union of the transversal intersections of smooth
divisors $\Sigma^S_{n-m+k}$ and $\Sigma^{S^c}_{n-m+k}$ where
$S\sqcup S^c=\{1,\cdots,n\}$ is a partition with $|S|=m$. In
particular, $Y_{n-m+k}$ is of codimension $2$ and the stabilizers
of points in $Y_{n-m+k}$ are all conjugates of $\CC^*$. The
weights of the action of the stabilizer $\CC^*$ on the normal
space to $Y_{n-m+k}$ are $2,-2$. By Luna's slice theorem
(\cite[Appendix 1.D]{MFK}), it follows that $F_{n-m+k}\git G$ is
smooth along the divisor $Y_{n-m+k}\git G$. If we let $\tFk\to
\Fk^{ss}$ be the blow-up of $\Fk^{ss}$ along $\Yk$,
$\tFk^{ss}=\tFk^s$ and $\tFk\git G=\tFk^s/G$ is nonsingular. Since
blow-up and quotient commute (\S\ref{sec2.2}), the induced map
$$\tFk\git G\to \Fk\git G$$ is a blow-up along $\Yk\git G$ which has
to be an isomorphism because the blow-up center is already a smooth
divisor. So we can use $\tFk^s$ instead of $\Fk^{ss}$ and apply the
same line of arguments as in the odd degree case. In this way, we
can establish Theorem \ref{thm1.1}.

To deduce Theorem \ref{thm1.2} from Theorem \ref{thm1.1}, we note
that by \cite[Corollary 3.5]{Simpson}, it suffices to prove that
$K_{\Mzek}+\alpha D_k$ is ample for $\frac{2}{m-k+2}<\alpha\le
\frac{2}{m-k+1}$ where $D_k=\Mzek-M_{0,n}$ is the boundary divisor
of $\Mzek$ (Proposition \ref{prop-amplerange}). By the
intersection number calculations of Alexeev and Swinarski
(\cite[\S3]{AlexSwin}), we obtain the nefness of $K_{\Mzek}+\alpha
D_k$ for $\alpha= \frac{2}{m-k+1}+s$ for some (sufficiently small)
positive number $s$. Because any positive linear combination of an
ample divisor and a nef divisor is ample, it suffices to show that
$K_{\Mzek}+\alpha D_k$ is ample for $\alpha=\frac{2}{m-k+2}+t$ for
\emph{any} sufficiently small $t>0$. We use induction on $k$. By
calculating the canonical divisor explicitly, it is easy to show
when $k=0$. Because $\varphi_k$ is a blow-up with exceptional
divisor $D^{m-k+1}_k$, $\varphi_k^*(K_{\Mzeko}+\alpha
D_{k-1})-\delta D^{m-k+1}_k$ is ample for small $\delta>0$ if
$K_{\Mzeko}+\alpha D_{k-1}$ is ample. By a direct calculation, we
find that these ample divisors give us $K_{\Mzek}+\alpha D_k$ with
$\alpha=\frac{2}{m-k+2}+t$ for any sufficiently small $t>0$. So we
obtain a proof of Theorem \ref{thm1.2}.

\bigskip

For the moduli spaces of \emph{unordered} weighted pointed stable
curves \[\widetilde{M}_{0,n\cdot\epsilon_k}=\Mzek/S_n\] we can
simply take the $S_n$ quotient of our sequence \eqref{eq1} and
thus $\widetilde{M}_{0,n\cdot\epsilon_k}$ can be constructed by a
sequence of \emph{weighted blow-ups} from $\PP^n\git
G=\left((\PP^1)^n\git G\right)/S_n$. In particular,
$\widetilde{M}_{0,n\cdot\epsilon_1}$ is a weighted blow-up of
$\PP^n\git G$ at its singular point when $n$ is even.

\bigskip

Here is an outline of this paper. In \S2, we recall necessary
materials about the moduli spaces $\Mzek$ of weighted pointed stable
curves, partial desingularization and blow-up along transversal
center. In \S3, we recall the blow-up construction of the moduli
space $F_k$ of weighted pointed parameterized stable curves. In \S4,
we prove Theorem \ref{thm1.1}. In \S5, we prove Theorem
\ref{thm1.2}. In \S6, we give a basis of the Picard group of $\Mzek$
as an application of Theorem \ref{thm1.1}.

\bigskip

\textbf{Acknowledgement.} This paper grew out of our effort to
prove a conjecture of Brendan Hassett (passed to us by David
Donghoon Hyeon): When $n$ is even,
$\widetilde{M}_{0,n\cdot\epsilon_1}$ is the (weighted) blow-up of
$\PP^n\git G$ at the singular point. It is our pleasure to thank
Donghoon Hyeon for useful discussions. We are also grateful to
David Smyth who kindly pointed out an error in a previous draft.

\section{Preliminaries}\label{sec2}

\subsection{Moduli of weighted pointed stable curves}\label{sec2.1}
We recall the definitions and basic facts on Hassett's moduli
spaces of weighted pointed stable curves from \cite{Hassett}.

A family of nodal curves of genus $g$ with $n$ marked points over
base scheme $B$ consists of
\begin{enumerate}
\item a flat proper
morphism $\pi : C \to B$ whose geometric fibers are nodal
connected curves of arithmetic genus $g$ and
\item sections $s_1, s_2, \cdots, s_n$ of $\pi$.
\end{enumerate}
An $n$-tuple
$\mathcal{A} =(a_1, a_2, \cdots, a_n) \in \mathbb{Q}^n$ with $0 <
a_i \le 1$ assigns a weight $a_i$ to the $i$-th marked point.
Suppose that $2g-2+a_1+a_2+\cdots+a_n > 0$.

\begin{definition}\cite[\S 2]{Hassett}
A family of nodal curves of genus $g$ with $n$ marked points $(C,
s_1, \cdots, s_n) \stackrel{\pi}{\to} B$ is stable of type $(g,
\mathcal{A})$ if
\begin{enumerate}
\item the sections $s_1, \cdots, s_n$ lie in the smooth locus of $\pi$;
\item for any subset
$\{s_{i_1}, \cdots, s_{i_r}\}$ of nonempty intersection, $a_{i_1}
+ \cdots + a_{i_r} \le 1$;
\item  $K_{\pi} + a_1s_1 + a_2s_2+\cdots+a_ns_n$ is $\pi$-relatively ample.
\end{enumerate}
\end{definition}

\begin{theorem}\cite[Theorem 2.1]{Hassett}
There exists a connected Deligne-Mumford stack
$\overline{\mathcal{M}}_{g, \mathcal{A}}$, smooth and proper over
$\ZZ$, representing the moduli functor of weighted pointed stable
curves of type $(g, \mathcal{A})$. The corresponding coarse moduli
scheme $\overline{M}_{g, \mathcal{A}}$ is projective over $\ZZ$.
\end{theorem}

When $g = 0$, there is no nontrivial automorphism for any weighted
pointed stable curve and hence $\overline{M}_{0, \mathcal{A}}$ is
a projective \emph{smooth variety} for any $\cA$.

There are natural morphisms between moduli spaces with different
weight data. Let $\mathcal{A}=\{a_1, \cdots, a_n\}$,
$\mathcal{B}=\{b_1, \cdots, b_n\}$ be two weight data and suppose
$a_i \ge b_i$ for all $1 \le i \le n$. Then there exists a
birational \emph{reduction} morphism
\[
    \varphi_{\mathcal{A}, \mathcal{B}}
    : \overline{\mathcal{M}}_{g, \mathcal{A}} \to
    \overline{\mathcal{M}}_{g, \mathcal{B}}.
\]
For $(C, s_1, \cdots, s_n) \in \overline{\mathcal{M}}_{g,
\mathcal{A}}$, $\varphi_{\mathcal{A}, \mathcal{B}}(C, s_1, \cdots,
s_n)$ is obtained by collapsing components of $C$ on which
$\omega_C + b_1s_1+ \cdots + b_ns_n$ fails to be ample. These
morphisms between moduli stacks induce corresponding morphisms
between coarse moduli schemes.

The exceptional locus of the reduction morphism
$\varphi_{\mathcal{A}, \mathcal{B}}$ consists of boundary divisors
$D_{I, I^c}$ where $I = \{i_1, \cdots, i_r\}$ and $I^c=\{j_1,
\cdots, j_{n-r}\}$ form a partition of $\{1, \cdots, n\}$
satisfying $r > 2$, $$a_{i_1} + \cdots + a_{i_r} > 1 \quad
\text{and}\quad b_{i_1} + \cdots + b_{i_r} \le 1.$$ Here $D_{I,
I^c}$ denotes the closure of the locus of $(C, s_1, \cdots, s_n)$
where $C$ has two irreducible components $C_1, C_2$ with $p_a(C_1)
= 0$, $p_a(C_2) = g$, $r$ sections $s_{i_1}, \cdots s_{i_r}$ lying
on $C_1$, and the other $n-r$ sections lying on $C_2$.

\begin{proposition}\cite[Proposition 4.5]{Hassett}\label{reduction}
The boundary divisor $D_{I,I^c}$ is isomorphic to $
\overline{M}_{0, \mathcal{A}'_I} \times \overline{M}_{g,
\mathcal{A}'_{I^c}},$ with $ \mathcal{A}'_I = (a_{i_1}, \cdots,
a_{i_r}, 1) $ and $\mathcal{A}'_{I^c}= (a_{j_1}, \cdots,
a_{j_{n-r}}, 1).$ Furthermore, $
    \varphi_{\mathcal{A}, \mathcal{B}}(D_{I, I^c})
    \cong \overline{M}_{g, \mathcal{B}'_{I^c}}$ with $ \mathcal{B}'_{I^c} = (b_{j_1}, \cdots, b_{j_{n-r}},
\sum_{k=1}^r    b_{i_k}).$
\end{proposition}

From now on, we focus on the $g=0$ case. Let $$m = \lfloor
\frac{n}{2}\rfloor,\quad \frac{1}{m-k+1} < \epsilon_k \le
\frac{1}{m-k}\quad \text{and}\quad n\cdot \epsilon_k = (\epsilon_k,
\cdots, \epsilon_k).$$ Consider the reduction morphism
\[
    \varphi_{n\cdot \epsilon_{k}, n \cdot \epsilon_{k-1}}
    : \overline{M}_{0, n \cdot \epsilon_{k}} \to
    \overline{M}_{0, n \cdot \epsilon_{k-1}}.
\]
Then $D_{I, I^c}$ is contracted by $\varphi_{n\cdot \epsilon_{k},
n \cdot \epsilon_{k-1}}$ if and only if $|I| = m - k + 1$.
Certainly, there are ${n} \choose {m-k+1}$ such partitions
$I\sqcup I^c$ of $\{1,\cdots,n\}$.

For two subsets $I, J \subset \{1, \cdots, n\}$ such that $|I| =
|J| = m - k + 1$, $D_{I, I^c} \cap D_{J, J^c}$ has codimension at
least two in $\Mzek$. So if we denote the complement of the
intersections of the divisors by
\[
    \Mzek' = \Mzek - \bigcup_{|I| = |J| = n-k+1, I \ne J} D_{I, I^c} \cap D_{J,
    J^c},
\]
we have $\mathrm{Pic}(\Mzek') = \mathrm{Pic}(\Mzek)$. The
restriction of $\varphi_{n \cdot \epsilon_{k}, n \cdot
\epsilon_{k-1}}$ to $\Mzek'$ is a contraction of ${n} \choose
{m-k+1}$ \emph{disjoint} divisors and its image is an open subset
whose complement has codimension at least two. Therefore we obtain
the following equality of Picard numbers:
\begin{equation}\label{eq-eqPicNumber}
    \rho(\overline{M}_{0, n \cdot \epsilon_{k}}) =
    \rho(\overline{M}_{0, n \cdot \epsilon_{k-1}}) + {n \choose {m-k+1}}.
\end{equation}

It is well known that the Picard number of $\Mzn$ is
\begin{equation}\label{eq-eqPic2} \rho(\Mzn)=\rho(\overline{M}_{0, n \cdot
\epsilon_{m-2}})=2^{n-1}-\binom{n}{2}-1.
\end{equation}
Hence we obtain the following lemma from \eqref{eq-eqPicNumber}
and \eqref{eq-eqPic2}.
\begin{lemma}\label{lem-PicNumMze}\begin{enumerate}\item
If $n$ is odd, $\rho(\Mzek)= n+\sum_{i=1}^k\binom{n}{m-i+1}$.\item
If $n$ is even,
$\rho(\Mzek)=n+\frac12\binom{n}{m}+\sum_{i=2}^k\binom{n}{m-i+1}$.\end{enumerate}
\end{lemma}

\subsection{Partial desingularization}\label{sec2.2}
We recall a few results from \cite{Kirwan, Hu} on change of
stability in a blow-up.

Let $G$ be a complex reductive group acting on a projective
nonsingular variety $X$. Let $L$ be a $G$-linearized ample line
bundle on $X$. Let $Y$ be a $G$-invariant closed subvariety of
$X$, and let $\pi : \widetilde{X} \to X$ be the blow-up of $X$
along $Y$, with exceptional divisor $E$. Then for sufficiently
large $d$, $L_d = \pi^*L^d \otimes \cO(-E)$ becomes very ample,
and there is a natural lifting of the $G$-action to $L_d$ (
\cite[\S3]{Kirwan}).

Let $X^{ss}$(resp. $X^s$) denote the semistable (resp. stable)
part of $X$. With respect to the polarizations $L$ and $L_d$, the
following hold (\cite[\S3]{Kirwan} or \cite[Theorem 3.11]{Hu})
:\begin{equation}\label{eq-StabBlowup} \widetilde{X}^{ss} \subset
\pi^{-1}(X^{ss}), \qquad \widetilde{X}^{s} \supset
\pi^{-1}(X^{s}).\end{equation} In particular, if $X^{ss} = X^s$,
then $\widetilde{X}^{ss} = \widetilde{X}^s = \pi^{-1}(X^s)$.

For the next lemma, let us suppose $Y^{ss}=Y\cap X^{ss}$ is
nonsingular. We can compare the GIT quotient of $\widetilde{X}$ by
$G$ with respect to $L_d$ with the quotient of $X$ by $G$ with
respect to $L$.
\begin{lemma}\cite[Lemma 3.11]{Kirwan} \label{blowupGIT}
For sufficiently large $d$, $\widetilde{X}\git G$ is the blow-up
of $X\git G$ along the image $Y \git G$ of $Y^{ss}$.
\end{lemma}

Let $\cI$ be the ideal sheaf of $Y$. In the statement of Lemma
\ref{blowupGIT}, the blow-up is defined by the ideal sheaf
$(\cI^m)_G$ which is the $G$-invariant part of $\cI^m$,
for some $m$. (See the proof of \cite[Lemma
3.11]{Kirwan}.) In the cases considered in this paper, the
blow-ups always take place along \emph{reduced} ideals, i.e.
$\widetilde{X}\git G$ is the blow-up of $X\git G$ along the
subvariety $Y\git G$ because of the following.
\begin{lemma}\label{lem-specialcasebl}
Let $G=SL(2)$ and $\CC^*$ be the maximal torus of $G$. Suppose
$Y^{ss}$ is smooth. The blow-up $\widetilde{X}\git G\to X\git G$
is the blow-up of the reduced ideal of $Y\git G$ if any of the
following holds:
\begin{enumerate} \item The stabilizers of points in $X^{ss}$ are
all equal to the center $\{\pm 1\}$, i.e. $\bar G=SL(2)/\{\pm 1\}$
acts on $X^{ss}$ freely.
\item If we denote the $\CC^*$-fixed locus in $X^{ss}$ by
$Z^{ss}_{\CC^*}$, $Y^{ss}=Y\cap X^{ss}=GZ^{ss}_{\CC^*}$ and the
stabilizers of points in $X^{ss}-Y^{ss}$ are all $\{\pm 1\}$.
Furthermore suppose that the weights of the action of $\CC^*$ on
the normal space of $Y^{ss}$ at any $y\in Z^{ss}_{\CC^*}$ are $\pm
l$ for some $l\ge 1$.
\item There exists a smooth divisor $W$ of $X^{ss}$ which
intersects transversely with $Y^{ss}$ such that the stabilizers of
points in $X^{ss}-W$ are all $\mathbb{Z}_2=\{\pm 1\}$ and the
stabilizers of points in $W$ are all isomorphic to $\mathbb{Z}_4$.
\end{enumerate}
In the cases (1) and (3), $Y\git G=Y^s/G$ and $X\git G=X^s/G$ are
nonsingular and the morphism $\widetilde{X}\git G\to X\git G$ is
the smooth blow-up along the smooth subvariety $Y\git G$.
\end{lemma}
\begin{proof} Let us consider the first case. Let $\bar
G=PGL(2)$. By Luna's \'etale slice theorem \cite[Appendix 1.D]{MFK},
\'etale locally near a point in $Y^{ss}$, $X^{ss}$ is
$\bar{G}\times S$ and $Y^{ss}$ is $\bar{G}\times S^Y$ for some
nonsingular locally closed subvariety $S$ and $S^Y=S\cap Y$. Then
\'etale locally $\widetilde{X}^{ss}$ is $\bar{G}\times
\mathrm{bl}_{S^Y}S$ where $\mathrm{bl}_{S^Y}S$ denotes the blow-up
of $S$ along the nonsingular variety $S^Y$. Thus the quotients
$X\git G$, $Y\git G$ and $\widetilde{X}\git G$ are \'etale locally
$S$, $S^Y$ and $\mathrm{bl}_{S^Y}S$ respectively. This implies
that the blow-up $\widetilde{X}\git G\to X\git G$ is the smooth
blow-up along the reduced ideal of $Y\git G$.

For the second case, note that the orbits in $Y^{ss}$ are closed
in $X^{ss}$ because the stabilizers are maximal. So we can again
use Luna's slice theorem to see that \'etale locally near a point
$y$ in $Y^{ss}$, the varieties $X^{ss}$, $Y^{ss}$ and
$\widetilde{X}$ are respectively $G\times_{\CC^*}S$,
$G\times_{\CC^*}S^0$ and $G\times_{\CC^*}\mathrm{bl}_{S^0}S$ for
some nonsingular locally closed $\CC^*$-equivariant subvariety $S$
and its $\CC^*$-fixed locus $S^0$. Therefore the quotients $X\git
G$, $Y\git G$ and $\widetilde{X}\git G$ are \'etale locally $S\git
\CC^*$, $S^0$ and $(\mathrm{bl}_{S^0}S)\git \CC^*$. Thus it
suffices to show
$$(\mathrm{bl}_{S^0}S)\git \CC^*\cong
\mathrm{bl}_{S^0}(S\git \CC^*).$$ Since $X$ is smooth, \'etale
locally we can choose our $S$ to be the normal space to the orbit
of $y$ and $S$ is decomposed into the weight spaces $S^0\oplus
S^+\oplus S^-$. As the action of $\CC^*$ extends to $SL(2)$, the
nonzero weights are $\pm l$ by assumption. If we choose
coordinates $x_1,\cdots, x_r$ for $S^+$ and $y_1,\cdots, y_s$ for
$S^-$, the invariants are polynomials of $x_iy_j$ and thus
$(I^{2m})_{\CC^*}=(I_{\CC^*})^m$ for $m\ge 1$ where $I=\langle
x_1,\cdots,x_r,y_1,\cdots,y_s \rangle$ is the ideal of $S^0$. By
\cite[II Exe. 7.11]{Hartshorne}, we have
$$\mathrm{bl}_{S^0}S=\mathrm{Proj}_S(\oplus_m I^m)\cong
\mathrm{Proj}_S(\oplus_m I^{2m})$$ and thus
$$(\mathrm{bl}_{S^0}S)\git \CC^*
=\mathrm{Proj}_{S\git \CC^*}(\oplus_m I^{2m})_{\CC^*}
=\mathrm{Proj}_{S\git \CC^*}\left(\oplus_m
(I_{\CC^*})^{m}\right)=\mathrm{bl}_{I_{\CC^*}}(S\git \CC^*).$$
Since $S$ is factorial and $I$ is reduced, $I_{\CC^*}$ is reduced.
(If $f^m\in I_{\CC^*}$, then $f\in I$ and $(g\cdot f)^m=f^m$ for
$g\in \CC^*$. By factoriality, $g\cdot f$ may differ from $f$ only
by a constant multiple, which must be an $m$-th root of unity.
Because $\CC^*$ is connected, the constant must be $1$ and hence
$f\in I_{\CC^*}$.) Therefore $I_{\CC^*}$ is the reduced ideal of
$S^0$ on $S\git \CC^*$ and hence $(\mathrm{bl}_{S^0}S)\git
\CC^*\cong \mathrm{bl}_{S^0}(S\git \CC^*)$ as desired.

The last case is similar to the first case. Near a point in $W$,
$X^{ss}$ is \'etale locally $\bar G\times_{\ZZ_2}S$ where
$S=S_W\times\CC$ for some smooth variety $S_W$. $\ZZ_2$ acts
trivially on $S_W$ and by $\pm 1$ on $\CC$. Etale locally $Y^{ss}$
is $\bar G\times_{\ZZ_2}S_Y$ where $S_Y=(S_W\cap Y)\times \CC$.
The quotients $X\git G$, $Y\git G$ and $\widetilde{X}\git G$ are
\'etale locally $S_W\times \CC$, $(S_W\cap Y)\times \CC$ and
$\mathrm{bl}_{S_W\cap Y}S_W\times \CC$. This proves our lemma.
\end{proof}
\begin{corollary}\label{cor-blcomquot}
Suppose that (1) of Lemma \ref{lem-specialcasebl} holds. If
$Y^{ss}=Y_1^{ss}\cup\cdots\cup Y_r^{ss}$ is a transversal union of
smooth subvarieties of $X^{ss}$ and if $\widetilde{X}$ is the
blow-up of $X^{ss}$ along $Y^{ss}$, then $\widetilde{X}\git G$ is
the blow-up of $X\git G$ along the reduced ideal of $Y\git G$
which is again a transversal union of smooth varieties $Y_i\git
G$. The same holds under the condition (3) of Lemma
\ref{lem-specialcasebl} if furthermore $Y_i$ are transversal to
$W$.
\end{corollary}
\begin{proof}
Because of the assumption (1), $X^{ss}=X^s.$ If
$Y^{ss}=Y_1^{ss}\cup\cdots\cup Y_r^{ss}$ is a transversal union of
smooth subvarieties of $X^{ss}$ and if $\pi:\widetilde{X}\to
X^{ss}$ is the blow-up along $Y^{ss}$, then
$\widetilde{X}^s=\widetilde{X}^{ss}=\pi^{-1}(X^s)$ is the
composition of smooth blow-ups along (the proper transforms of)
the irreducible components $Y_i^{ss}$ by Proposition
\ref{prop-bltrcen} below. For each of the smooth blow-ups, the
quotient of the blown-up space is the blow-up of the quotient
along the reduced ideal of the quotient of the center by Lemma
\ref{lem-specialcasebl}. Hence $\widetilde{X}\git G\to X\git G$ is
the composition of smooth blow-ups along irreducible smooth
subvarieties which are proper transforms of $Y_i\git G$. Hence
$\widetilde{X}\git G$ is the blow-up along the union $Y\git G$ of
$Y_i\git G$ by Proposition \ref{prop-bltrcen} again.

The case (3) of Lemma \ref{lem-specialcasebl} is similar and we
omit the detail.
\end{proof}

Finally we recall Kirwan's partial desingularization construction
of GIT quotients. Suppose $X^{ss} \ne X^s$ and $X^s$ is nonempty.
Kirwan in \cite{Kirwan} introduced a systematic way of blowing up
$X^{ss}$ along a sequence of nonsingular subvarieties to obtain a
variety $\widetilde{X}$ with linearized $G$ action such that
$\widetilde{X}^{ss} = \widetilde{X}^s$ and $\widetilde{X}\git G$
has at worst finite quotient singularities only, as
follows:\begin{enumerate}
\item Find a maximal dimensional connected reductive subgroup $R$
such that the $R$-fixed locus $Z_R^{ss}$ in $X^{ss}$ is nonempty.
Then $$GZ_R^{ss}\cong G\times_{N^R}Z_R^{ss}$$ is a nonsingular
closed subvariety of $X^{ss}$ where $N^R$ denotes the normalizer
of $R$ in $G$.
\item Blow up $X^{ss}$ along $GZ_R^{ss}$ and find the semistable
part $X_1^{ss}$. Go back to step 1 and repeat this precess until
there are no more strictly semistable points.
\end{enumerate}
Kirwan proves that this process stops in finite steps and
$\widetilde{X}\git G$ is called the \emph{partial
desingularization} of $X \git G$. We will drop ``partial" if it is
nonsingular.

\subsection{Blow-up along transversal center}\label{sec2.3} We show that the
blow-up along a center whose irreducible components are
transversal smooth varieties is isomorphic to the result of smooth
blow-ups along the irreducible components in any order. This fact
can be directly proved but instead we will see that it is an easy
special case of beautiful results of L. Li in \cite{Li}.

\begin{definition} \cite[\S1]{Li}
(1) For a nonsingular algebraic variety $X$, an \emph{arrangement}
of subvarieties $S$ is a finite collection of nonsingular
subvarieties such that all nonempty scheme-theoretic intersections
of subvarieties in $S$ are again in $S$.

(2) For an arrangement $S$, a subset $B\subset S$ is called a
\emph{building set} of $S$ if for any $s \in S- B$, the minimal
elements in $\{b \in B : b \supset s\}$ intersect transversally
and the intersection is $s$.

(3) A set of subvarieties $B$ is called a \emph{building set} if
all the possible intersections of subvarieties in $B$ form an
arrangement $S$ (called the induced arrangement of $B$) and $B$ is
a building set of $S$.
\end{definition}

The \emph{wonderful compactification} $X_B$ of $X^0=X-\cup_{b\in
B} b$ is defined as the closure of $X^0$ in $\prod_{b\in
B}\mathrm{bl}_bX$. Li then proves the following.
\begin{theorem}\cite[Theorem 1.3]{Li} \label{thm-Li1}
Let $X$ be a nonsingular variety and $B = \{b_1, \cdots,b_n\}$ be
a nonempty building set of subvarieties of $X$. Let $I_i$ be the
ideal sheaf of $b_i \in B$. \begin{enumerate}\item The wonderful
compactification $X_B$ is isomorphic to the blow-up of $X$ along
the ideal sheaf $I_1I_2\cdots I_n$. \item If we arrange $B =
\{b_1, \cdots,b_n\}$ in such an order that the first $i$ terms
$b_1,\cdots,b_i$ form a building set for any $1\le i\le n$, then
$X_B = \mathrm{bl}_{\tilde{b}_n} \cdots \mathrm{bl}_{\tilde{b}_2}
\mathrm{bl}_{b_1} X$, where each blow-up is along a nonsingular
subvariety $\tilde{b}_i$.\end{enumerate}
\end{theorem}
Here $\tilde{b}_i$ is the \emph{dominant transform} of $b_i$ which
is obtained by taking the proper transform when it doesn't lie in
the blow-up center or the inverse image if it lies in the center,
in each blow-up. (See \cite[Definition 2.7]{Li}.)

Let $X$ be a smooth variety and let $Y_1, \cdots, Y_n$ be
transversally intersecting smooth closed subvarieties. Here,
\emph{transversal intersection} means that for any nonempty $S
\subset \{1, \cdots, n\}$ the intersection $Y_S:=\cap_{i \in
S}Y_i$ is smooth and the normal bundle $N_{Y_S/X}$ in $X$ of $Y_S$
is the direct sum of the restrictions of the normal bundles
$N_{Y_i/X}$ in $X$ of $Y_i$, i.e.
$$N_{Y_S/X} = \bigoplus_{i\in S} N_{Y_i/X}|_{Y_S}.$$
If we denote the ideal of $Y_i$ by $I_i$, the ideal of the union
$\cup_{i=1}^n Y_i$ is the product $I_1I_2\cdots I_n$. Moreover for
any permutation $\tau\in S_n$ and $1\le i\le n$,
$B=\{Y_{\tau(1)},\cdots,Y_{\tau(i)}\}$ is clearly a building set.
By Theorem \ref{thm-Li1} we obtain the following.
\begin{proposition}\label{prop-bltrcen}
Let $Y=Y_1\cup\cdots \cup Y_n$ be a union of transversally
intersecting smooth subvarieties of a smooth variety $X$. Then the
blow-up of $X$ along $Y$ is isomorphic to
\[
\mathrm{bl}_{\tilde Y_{\tau(n)}}\cdots \mathrm{bl}_{\tilde
Y_{\tau(2)}}\mathrm{bl}_{Y_{\tau(1)}} X
\]
for any permutation $\tau\in S_n$ where $\tilde{Y}_i$ denotes the
proper transform of $Y_i$.
\end{proposition}


\subsection{Log canonical model}
Let $X$ be a normal projective variety and $D = \sum a_i D_i$ be a
rational linear combination of prime divisors of $X$ with $0 < a_i
\le 1$. A \emph{log resolution} of $(X, D)$ is a birational
morphism $\pi : Y \to X$ from a smooth projective variety $Y$ to
$X$ such that $\pi^{-1}(D_i)$ and the exceptional divisors $E_i$
of $\pi$ are simple normal crossing divisors on $Y$. Then the
discrepancy formula
\[
    K_Y + \pi^{-1}_* (D) \equiv
    \pi^*(K_X + D) + \sum_{E_i : \mbox{exceptional}} a(E_i, X, D)E_i,
\]
defines the \emph{discrepancy} of $(X, D)$ by
\[
    \mathrm{discrep}(X,D) := \inf \{ a(E, X, D) : E : \mbox{exceptional}\}.
\]

Let $(X, D)$ be a pair where $X$ is a normal projective variety
and $D = \sum a_i D_i$ be a rational linear combination of prime
divisors with $0 < a_i \le 1$. Suppose that $K_X + D$ is
$\QQ$-Cartier. A pair $(X, D)$ is \emph{log canonical (abbrev.
lc)} if $\mathrm{discrep}(X,D) \ge -1$ and \emph{Kawamata log
terminal (abbrev. klt)} if $\mathrm{discrep}(X,D) > -1$ and
$\lfloor D \rfloor \le 0$.

When $X$ is smooth and $D$ is a normal crossing effective divisor,
$(X, D)$ is always lc and  is klt if all $a_i < 1$.

\begin{definition}
For lc pair $(X, D)$, the \emph{canonical ring} is
\[
    R(X, K_X + D) := \oplus_{l \ge 0} H^0(X, \cO_X(\lfloor l (K_X + D) \rfloor))
\]
and the \emph{log canonical model} is
\[
    \proj R(X, K_X + D).
\]
\end{definition}
In \cite{BCHM}, Birkar, Cascini, Hacon and McKernan proved that
for any klt pair $(X, D)$, the canonical ring is finitely
generated, so the log canonical model always exists.

\section{Moduli of weighted parameterized stable
curves}\label{sec3}

Let $X$ be a smooth projective variety. In this section, we
decompose the map $$X[n]\to X^n$$ defined by Fulton and MacPherson
(\cite{FM}) into a \emph{symmetric} sequence of blow-ups along
transversal centers. A. Mustata and M. Mustata already considered
this problem in their search for intermediate moduli spaces for
the stable map spaces in \cite[\S1]{Mustata}. Let us recall their
construction.

\bigskip

\noindent \textbf{Stage 0}: Let $F_0=X^n$ and $\Gamma_0=X^n\times
X$. For a subset $S$ of $\{1,2,\cdots,n\}$, we let
\[
\Sigma^S_0=\{(x_1,\cdots,x_n)\in X^n\,|\, x_i=x_j \text{ if }
i,j\in S\}, \quad \Sigma^k_0=\cup_{|S|=k}\Sigma_0^S
\]
and let  $\sigma^i_0\subset \Gamma_0$ be the graph of the $i$-th
projection $X^n\to X$. Then $\Sigma_0^n\cong X$ is a smooth
subvariety of $F_0$. For each $S$, fix any $i_S\in S$.

\bigskip

\noindent \textbf{Stage $1$}: Let $F_1$ be the blow-up of $F_0$
along $\Sigma_0^n$. Let $\Sigma_1^n$ be the exceptional divisor
and $\Sigma_1^S$ be the proper transform of $\Sigma_0^S$ for
$|S|\ne n$. Let us define $\Gamma_1$ as the blow-up of
$F_1\times_{F_0}\Gamma_0$ along $\Sigma^n_1\times_{F_0}\sigma^1_0$
so that we have a flat family
\[
\Gamma_1\to F_1\times_{F_0}\Gamma_0 \to F_1
\]
of varieties over $F_1$. Let $\sigma_1^i$ be the proper transform
of $\sigma_0^i$ in $\Gamma_1$. Note that $\Sigma^S_{1}$ for
$|S|=n-1$ are all disjoint smooth varieties of same dimension.

\bigskip

\noindent \textbf{Stage $2$}: Let $F_2$ be the blow-up of $F_1$
along $\Sigma_1^{n-1}=\sum_{|S|=n-1}\Sigma_1^S$. Let $\Sigma_2^S$
be the exceptional divisor lying over $\Sigma_1^S$ if $|S|=n-1$
and $\Sigma_2^S$ be the proper transform of $\Sigma_1^S$ for
$|S|\ne n-1$. Let us define $\Gamma_2$ as the blow-up of
$F_2\times_{F_1}\Gamma_1$ along the disjoint union of
$\Sigma^S_2\times_{F_1}\sigma^{i_S}_1$ for all $S$ with $|S|=n-1$
so that we have a flat family
\[
\Gamma_2\to F_2\times_{F_1}\Gamma_1 \to F_2
\]
of varieties over $F_2$. Let $\sigma_2^i$ be the proper transform
of $\sigma_1^i$ in $\Gamma_2$. Note that $\Sigma^S_{2}$ for
$|S|=n-2$ in $F_2$ are all transversal smooth varieties of same
dimension. Hence the blow-up of $F_2$ along their union is smooth
by \S\ref{sec2.3}.

\bigskip

We can continue this way until we reach the last stage.

\bigskip
\noindent \textbf{Stage $n-1$}: Let $F_{n-1}$ be the blow-up of
$F_{n-2}$ along $\Sigma_{n-2}^2=\sum_{|S|=2}\Sigma_{n-2}^S$. Let
$\Sigma_{n-1}^S$ be the exceptional divisor lying over
$\Sigma_{n-2}^S$ if $|S|=2$ and $\Sigma_{n-1}^S$ be the proper
transform of $\Sigma_{n-2}^S$ for $|S|\ne 2$. Let us define
$\Gamma_{n-1}$ as the blow-up of
$F_{n-1}\times_{F_{n-2}}\Gamma_{n-2}$ along the disjoint union of
$\Sigma^S_{n-1}\times_{F_{n-2}}\sigma^{i_S}_{n-2}$ for all $S$
with $|S|=2$ so that we have a flat family
\[
\Gamma_{n-1}\to F_{n-1}\times_{F_{n-2}}\Gamma_{n-2} \to F_{n-1}
\]
of varieties over $F_{n-1}$. Let $\sigma_{n-1}^i$ be the proper
transform of $\sigma_{n-2}^i$ in $\Gamma_{n-1}$.

\bigskip
Nonsingularity of the blown-up spaces $F_k$ are guaranteed by the
following.

\begin{lemma}\label{lem3-1}
$\Sigma^S_{k}$ for $|S|\ge n-k$ are transversal in $F_{k}$ i.e.
the normal bundle in $F_{k}$ of  the intersection
$\cap_i\Sigma^{S_i}_k$ for distinct $S_i$ with $|S_i|\ge n-k$ is
the direct sum of  the restriction of the normal bundles in $F_k$
of $\Sigma^{S_i}_k$.
\end{lemma}
\begin{proof}
This is a special case of the inductive construction of the
wonderful compactification in \cite{Li}. (See \S \ref{sec2.3}.) In
our situation, the building set is the set of all diagonals $B_0 =
\{\Sigma_0^S | S \subset \{1, 2, \cdots, n\}\}$. By
\cite[Proposition 2.8]{Li}, $B_k=\{\Sigma_k^S\}$ is a building set
of an arrangement in $F_k$ and hence the desired transversality
follows.
\end{proof}

By construction, $F_k$ are all smooth and $\Gamma_k\to F_k$ are
equipped with $n$ sections $\sigma_k^i$. When $\dim X=1$,
$\Sigma^2_{n-2}$ is a divisor and thus $F_{n-1}=F_{n-2}$. In
\cite[Proposition 1.8]{Mustata}, Mustata and Mustata prove that
the varieties $F_k$ are fine moduli spaces for some moduli
functors as follows.

\begin{definition}\label{def3.1} \cite[Definition 1.7]{Mustata}
A family of \emph{$k$-stable parameterized rational curves} over
$S$ consists of a flat family of curves $\pi:C\to S$, a morphism
$\phi:C\to S\times \PP^1$ of degree 1 over each geometric fiber
$C_s$ of $\pi$ and $n$ marked sections $\sigma^1,\cdots, \sigma^n$
of $\pi$ such that for all $s\in S$, \begin{enumerate}
\item no more than $n-k$ of the marked points $\sigma^i(s)$ in $C_s$ coincide;
\item any ending irreducible curve in $C_s$, except the parameterized one, contains
more than $n-k$ marked points;
\item all the marked points are smooth points of the curve $C_s$;
\item $C_s$ has finitely many automorphisms preserving the marked
points and the map to $\PP^1$.
\end{enumerate}
\end{definition}

\begin{proposition}\label{prop3.2} \cite[Proposition 1.8]{Mustata}
Let $X=\PP^1$. The smooth variety $F_k$ finely represents the
functor of isomorphism classes of families of $k$-stable
parameterized rational curves. In particular, $F_{n-2}=F_{n-1}$ is
the Fulton-MacPherson space $\PP^1[n]$.
\end{proposition}


\section{Blow-up construction of moduli of pointed stable
curves}\label{sec4}

In the previous section, we decomposed the natural map
$\PP^1[n]\to (\PP^1)^n$ of the Fulton-MacPherson space into a
sequence
\begin{equation}\label{eq4-1}\xymatrix{
\PP^1[n]\ar@{=}[r] & F_{n-2}\ar[r]^{\psi_{n-2}} &
F_{n-3}\ar[r]^{\psi_{n-3}} & \cdots \ar[r]^{\psi_2} &
F_1\ar[r]^{\psi_1} & F_0\ar@{=}[r] & (\PP^1)^n }\end{equation} of
blow-ups along transversal centers. By construction the morphisms
above are all equivariant with respect to the action of $G=SL(2)$.
For GIT stability, we use the \emph{symmetric} linearization
$L_0=\cO(1,\cdots,1)$ for $F_0$. For $F_k$ we use the
linearization $L_k$ inductively defined by
$L_k=\psi_k^*L_{k-1}\otimes \cO(-\delta_kE_k)$ where $E_k$ is the
exceptional divisor of $\psi_k$ and $\{\delta_k\}$ is a decreasing
sequence of sufficiently small positive numbers. Let
$m=\lfloor\frac{n}{2}\rfloor$. In this section, we prove the
following.

\begin{theorem}\label{thm4-1} (i) The GIT quotient
$F_{n-m+k}\git G$ for $1\le k\le m-2$ is
isomorphic to Hassett's moduli space of weighted pointed stable
rational curves $\Mzek$ with weights $n\cdot
\epsilon_k=(\epsilon_k,\cdots,\epsilon_k)$ where
$\frac1{m+1-k}<\epsilon_k\le \frac1{m-k}$. The induced maps on
quotients \[ \Mzek = F_{n-m+k}\git G \to F_{n-m+k-1}\git G
=\Mzeko\] are blow-ups along transversal centers for $k=2,\cdots,
m-2$.

(ii) If $n$ is odd, $$F_{m+1}\git G=\cdots =F_0\git
G=(\PP^1)^n\git G=\Mzez$$ and we have a sequence of blow-ups
\[
\Mzn=\Mzemt\to \Mzemth \to \cdots \to \Mzeo\to \Mzez =
(\PP^1)^n\git G
\]
whose centers are transversal unions of equidimensional smooth
varieties.

(iii) If $n$ is even, $\Mzeo$ is a desingularization of
$$(\PP^1)^n\git G=F_0\git G=\cdots =F_m\git G,$$
obtained by blowing up $\frac12\binom{n}{m}$
singular points so that we have a sequence of blow-ups
\[
\Mzn=\Mzemt\to \Mzemth \to \cdots \to \Mzeo\to (\PP^1)^n\git G.
\]
\end{theorem}
\begin{remark}
(1) When $n$ is even, $\Mzez$ is not defined because the sum of
weights does not exceed $2$.

(2) When $n$ is even, $\Mzeo$ is Kirwan's (partial)
desingularization of the GIT quotient $(\PP^1)^n\git G$ with
respect to the symmetric linearization $L_0=\cO(1,\cdots,1)$.
\end{remark}

Let $F_k^{ss}$ (resp. $F_k^s$) denote the semistable (resp. stable)
part of $F_k$. By \eqref{eq-StabBlowup}, we have
\begin{equation}\label{eq4-2}
\psi_k(F_k^{ss})\subset F_{k-1}^{ss},\qquad
\psi_k^{-1}(F_{k-1}^s)\subset F_k^s.
\end{equation}
Also recall from \cite{K1} that $x=(x_1,\cdots,x_n)\in (\PP^1)^n$
is semistable (resp. stable) if $> \frac{n}2$ (resp. $\ge
\frac{n}2$) of $x_i$'s are not allowed to coincide. In particular,
when $n$ is odd, $\psi_k^{-1}(F_{k-1}^s)=F_k^s=F_k^{ss}$ for all
$k$ and
\begin{equation}\label{eq4-3}
F_{m+1}^s=F_{m}^s=\cdots =F_0^s,
\end{equation}
because the blow-up centers lie in the unstable part. Therefore we
have
\begin{equation}\label{eq4-4}
F_{m+1}\git G = \cdots =F_0\git G= (\PP^1)^n\git G.
\end{equation}
When $n$ is even, $\psi_k$ induces a morphism $F_k^{ss}\to
F_{k-1}^{ss}$ and we have
\begin{equation}\label{eq4-5}
F_m^{ss}=F_{m-1}^{ss}=\cdots =F_0^{ss} \quad \text{and}\quad F_m\git
G=\cdots =F_0\git G=(\PP^1)^n\git G.
\end{equation}

Let us consider the case where $n$ is odd first. By forgetting the
parameterization of the parameterized component of each member of
family $(\Gamma_{m+k+1}\to F_{m+k+1},\sigma_{m+k+1}^i)$, we get a
rational map $F_{m+k+1}\dashrightarrow \Mzek$ for $k=0,1,\cdots,
m-2$. By the definition of the stability in \S\ref{sec2.1}, a
fiber over $\xi\in F_{m+k+1}$ is not stable with respect to
$n\cdot \epsilon_k=(\epsilon_k,\cdots,\epsilon_k)$ if and only if,
in each irreducible component of the curve, the number $a$ of
nodes and the number $b$ of marked points satisfy
$b\epsilon_k+a\le 2$. Obviously this cannot happen on the (GIT)
stable part $F_{m+k+1}^s$. Therefore we obtain a morphism
$F_{m+k+1}^s\to \Mzek$. By construction this morphism is
$G$-invariant and thus induces a morphism
$$\phi_k:F_{m+k+1}\git G\to \Mzek.$$ Since the stabilizer groups
in $G$ of points in $F_0^s$ are all $\{ \pm 1\}$, the quotient
$$\bar{\psi}_{m+k+1}:F_{m+k+1}\git G\to F_{m+k}\git G$$ of
$\psi_{m+k+1}$ is also a blow-up along a center which consists of
transversal smooth varieties by Corollary \ref{cor-blcomquot}.

Since the blow-up center has codimension $\ge 2$, the Picard
number increases by $\binom{n}{m-k+1}$ for $k=1, \cdots, m-2$.
Since the character group of $SL(2)$ has no free part, by the
descent result in \cite{DN}, the Picard number of $F_{m+1}\git
G=F_0^s/G$ is the same as the Picard number of $F_0^s$ which
equals the Picard number of $F_0$. Therefore $\rho(F_{m+1}\git
G)=n$ and the Picard number of $F_{m+k+1}\git G$ is
\[
n+\sum_{i=1}^{k}\binom{n}{m-i+1}
\]
which equals the Picard number of $\Mzek$ by Lemma
\ref{lem-PicNumMze}. Since $\Mzek$ and $F_{m+k+1}\git G$ are
smooth and their Picard numbers coincide, we conclude that
$\phi_k$ is an isomorphism as we desired. So we proved Theorem
\ref{thm4-1} for odd $n$.

Now let us suppose $n$ is even. For ease of understanding, we
divide our proof into several steps.

\bigskip
\noindent \underline{Step 1:} For $k\ge 1$, $F_{m+k}\git G$ are
nonsingular and isomorphic to the partial desingularizations
$\tilde{F}_{m+k}\git G$.

\bigskip
The GIT quotients $F_{m+k}\git G$ may be singular because there
are $\CC^*$-fixed points in the semistable part $F_{m+k}^{ss}$. So
we use Kirwan's partial desingularization of the GIT quotients
$F_{m+k}\git G$ (\S\ref{sec2.2}). The following lemma says that
the partial desingularization process has no effect on the
quotient $F_{m+k}\git G$ for $k\ge 1$.

\begin{lemma}\label{lem4-3}
Let $F$ be a smooth projective variety with linearized $G=SL(2)$
action and let $F^{ss}$ be the semistable part. Fix a maximal
torus $\CC^*$ in $G$. Let $Z$ be the set of $\CC^*$-fixed points
in $F^{ss}$. Suppose the stabilizers of all points in the stable
part $F^{s}$ are $\{\pm 1\}$ and $Y=GZ$ is the union of all closed
orbits in $F^{ss}-F^s$. Suppose that the stabilizers of points in
$Z$ are precisely $\CC^*$. Suppose further that $Y=GZ$ is of
codimension $2$. Let $\tilde{F}\to F^{ss}$ be the blow-up of
$F^{ss}$ along $Y$ and let $\tilde{F}^s$ be the stable part in
$\tilde{F}$ with respect to a linearization as in \S\ref{sec2.2}.
Finally suppose that for each $y\in Z$, the weights of the $\CC^*$
action on the normal space to $Y$ is $\pm l$ for some $l>0$. Then
$\tilde{F}\git G=\tilde{F}^s/G\cong F\git G$ and $F\git G$ is
nonsingular.
\end{lemma}
\begin{proof}
Since $\bar G=G/\{\pm 1\}$ acts freely on $F^s$, $F^s/G$ is
smooth. By assumption, $Y$ is the union of all closed orbits in
$F^{ss}-F^s$ and hence $F\git G-F^s/G=Y/G$. By Lemma
\ref{lem-specialcasebl} (2), $\tilde{F}^s/G$ is the blow-up of
$F\git G$ along the reduced ideal of $Y/G$. By our assumption, $Z$
is of codimension $4$ and
$$Y/G=GZ/G\cong G\times _{N^{\CC^*}}Z/G\cong Z/\ZZ_2$$
where $N^{\CC^*}$ is the normalizer of $\CC^*$ in $G$. Since the
dimension of $F\git G$ is $\dim F-3$, the blow-up center $Y/G$ is
nonsingular of codimension $1$. By Luna's slice theorem
(\cite[Appendix 1.D]{MFK}), the singularity of $F\git G$ at any
point $[Gy]\in Y/G$ is $\CC^2\git \CC^*$ where the weights are
$\pm l$. Obviously this is smooth and hence $F\git G$ is smooth
along $Y/G$. Since the blow-up center is a smooth divisor, the
blow-up map $\tilde{F}^s/G\to F\git G$ has to be an isomorphism.
\end{proof}

Let $Z_{m+k}$ be the $\CC^*$-fixed locus in $F_{m+k}^{ss}$ and let
$Y_{m+k}=GZ_{m+k}$. Then $Y_{m+k}$ is the disjoint union of
$$\Sigma_{m+k}^{S,S^c}:=\Sigma_{m+k}^S\cap \Sigma_{m+k}^{S^c}\cap
F_{m+k}^{ss} \quad \text{for }|S|=m, S^c=\{1,\cdots,n\}-S $$ which
are nonsingular of codimension $2$ for $k\ge 1$ by Lemma
\ref{lem3-1}. For a point $$(C,p_1,\cdots,p_n,f:C\to\PP^1)\in
\Sigma^{S,S^c}_{m+k},$$ the parameterized component of $C$ (i.e.
the unique component which is not contracted by $f$) has two nodes
and no marked points. The normal space $\CC^2$ to
$\Sigma^{S,S^c}_{m+k}$ is given by the smoothing deformations of
the two nodes and hence the stabilizer $\CC^*$ acts with weights
$2$ and $-2$.

The blow-up $\tilde{F}_{m+k}$ of $F_{m+k}^{ss}$ along $Y_{m+k}$
has no strictly semistable points by \cite[\S6]{Kirwan}. In fact,
the unstable locus in $\tilde{F}_{m+k}$ is the proper transform of
$\Sigma^S_{m+k}\cup \Sigma^{S^c}_{m+k}$ and the stabilizers of
points in $\tilde{F}^s_{m+k}$ are either $\ZZ_2=\{\pm 1\}$ (for
points not in the exceptional divisor of $\tilde{F}^s_{m+k}\to
F^{ss}_{m+k}$) or $\ZZ_4=\{\pm 1,\pm i\}$ (for points in the
exceptional divisor). Therefore, by Lemma \ref{lem4-3} and Lemma
\ref{lem-specialcasebl} (3), we have isomorphisms
\begin{equation}\label{eq4-10}
\tilde{F}_{m+k}^s/G\cong F_{m+k}\git G \end{equation} and
$F_{m+k}\git G$ are nonsingular for $k\ge 1$.

\bigskip
\noindent \underline{Step 2:} The partial desingularization
$\tilde{F}_m\git G$ is a nonsingular variety obtained by blowing
up the $\frac12\binom{n}{m}$ singular points of $F_m\git
G=(\PP^1)^n\git G$.

\bigskip
Note that $Y_m$ in $F_m^{ss}$ is the disjoint union of
$\frac12\binom{n}{m}$ orbits $\Sigma_m^{S,S^c}$ for $|S|=m$. By
Lemma \ref{lem-specialcasebl} (2), the morphism
$\tilde{F}_m^s/G\to F_m\git G$ is the blow-up at the
$\frac12\binom{n}{m}$ points given by the orbits of the blow-up
center. A point in $\Sigma^{S,S^c}_{m}$ is represented by
$(\PP^1,p_1,\cdots,p_n,\mathrm{id})$ with $p_i=p_j$ if $i,j\in S$
or $i,j\in S^c$. Without loss of generality, we may let
$S=\{1,\cdots, m\}$. The normal space to an orbit
$\Sigma^{S,S^c}_{m}$ is given by
\[
(T_{p_1}\PP^1)^{m-1}\times
(T_{p_{m+1}}\PP^1)^{m-1}=\CC^{m-1}\times\CC^{m-1}
\]
and $\CC^*$ acts with weights $2$ and $-2$ respectively on the two
factors. By Luna's slice theorem, \'etale locally near
$\Sigma^{S,S^c}_{m}$, $F_m^{ss}$ is
$G\times_{\CC^*}(\CC^{m-1}\times\CC^{m-1})$ and $\tilde{F}_m$ is
$G\times_{\CC^*}\mathrm{bl}_{0}(\CC^{m-1}\times\CC^{m-1})$ while
$\tilde{F}_m^s$ is
$G\times_{\CC^*}\left[\mathrm{bl}_{0}(\CC^{m-1}\times\CC^{m-1})
-\mathrm{bl}_{0}\CC^{m-1}\sqcup \mathrm{bl}_{0}\CC^{m-1}\right]$.
By an explicit local calculation, the stabilizers of points on the
exceptional divisor of $\widetilde{F}_m$ are $\ZZ_4=\{\pm 1,\pm
i\}$ and the stabilizers of points over $F_m^s$ are $\ZZ_2=\{\pm
1\}$. Since the locus of nontrivial stabilizers for the action of
$\bar G$ on $\tilde{F}^s_m$ is a smooth divisor with stabilizer
$\ZZ_2$, $\tilde{F}_m\git G=\tilde{F}^s_m/G$ is smooth and hence
$\tilde{F}_m^s/G$ is the desingularization of $F_m\git G$ obtained
by blowing up its $\frac12\binom{n}{m}$ singular points.

\bigskip
\noindent \underline{Step 3:} The morphism
$\bar\psi_{m+k+1}:F_{m+k+1}\git G\to F_{m+k}\git G$ is the blow-up
along the union of transversal smooth subvarieties for $k\ge 1$.
For $k=0$, we have $\tilde{F}^s_{m+1}=\tilde{F}^s_m$ and thus
$$F_{m+1}\git G\cong
\tilde{F}^s_{m+1}/G=\tilde{F}^s_m/G=\tilde{F}_m\git G$$ is the
blow-up along its $\frac12\binom{n}{m}$ singular points.

\bigskip
From Lemma \ref{lem3-1}, we know $\Sigma^S_{m+k}$ for $|S|\ge m-k$
are transversal in $F_{m+k}$. In particular,
$$\bigcup_{|S|=m}\Sigma_{m+k}^S\cap \Sigma_{m+k}^{S^c}$$ intersects transversely with the
blow-up center $$\bigcup_{|S'|=m-k} \Sigma^{S'}_{m+k}$$ for
$\psi_{m+k+1}:F_{m+k+1}\to F_{m+k}$. Hence, by Proposition
\ref{prop-bltrcen} we have a commutative diagram
\begin{equation}\label{eq4-12}
\xymatrix{ \tilde{F}_{m+k+1}\ar[r]\ar[d] & \tilde{F}_{m+k}\ar[d]\\
F_{m+k+1}^{ss}\ar[r] & F_{m+k}^{ss} }
\end{equation}
for $k\ge 1$ where the top horizontal arrow is the blow-up along
the proper transforms $\tilde{\Sigma}^{S'}_{m+k}$ of
$\Sigma^{S'}_{m+k}$, $|S'|=m-k$. By Corollary \ref{cor-blcomquot},
we deduce that for $k\ge 1$, $\bar\psi_{m+k+1}$ is the blow-up
along the transversal union of smooth subvarieties
$\tilde{\Sigma}^{S'}_{m+k}\git G\cong \Sigma^{S'}_{m+k}\git G$.

For $k=0$, the morphism $\tilde{F}_{m+1}\to \tilde{F}_m$ is the
blow-up along the proper transforms of $\Sigma_m^S$ and
$\Sigma_m^{S^c}$ for $|S|=m$. But these are unstable in
$\tilde{F}_m$ and hence the morphism $\tilde{F}_{m+1}^s\to
\tilde{F}_m^s$ on the stable part is the identity map. So we
obtain $\tilde{F}_{m+1}^s= \tilde{F}_m^s$ and
$\tilde{F}_{m+1}^s/G\cong \tilde{F}_m^s/G$.

\bigskip
\noindent \underline{Step 4:} Calculation of Picard numbers.

\bigskip
The Picard number of $F_m^{ss}=F_0^{ss}\subset F_0=(\PP^1)^n$ is
$n$ and so the Picard number of $\tilde{F}_m$ is
$n+\frac12\binom{n}{m}$. By the descent lemma of \cite{DN} as in
the odd degree case, the Picard number of
$$F_{m+1}\git G\cong \tilde{F}_{m+1}^s/G=\tilde{F}_m^s/G$$ equals the
Picard number $n+\frac12\binom{n}{m}$ of $\tilde{F}_m^s$. Since
the blow-up center of $\tilde{F}_{m+k}\git G\to
\tilde{F}_{m+k-1}\git G$ has $\binom{n}{m-k+1}$ irreducible
components, the Picard number of $\tilde{F}_{m+k}\git G\cong
F_{m+k}\git G$ is
\begin{equation}\label{eq4-13}
n+\frac12\binom{n}{m}+\sum_{i=2}^{k}\binom{n}{m-i+1}
\end{equation}
for $k\ge 2$.

\bigskip
\noindent \underline{Step 5:} Completion of the proof.

\bigskip
As in the odd degree case, for $k\ge 1$ the universal
family $\pi_k:\Gamma_{m+k}\to F_{m+k}$ gives rise to a family of
pointed curves by considering the linear system
$K_{\pi_k}+\epsilon_k\sum_i\sigma^i_{m+k}$. Over the semistable
part $F_{m+k}^{ss}$ it is straightforward to check that this gives
us a family of $n\cdot \epsilon_k$-stable pointed curves.
Therefore we obtain an invariant morphism $$F_{m+k}^{ss}\to
\Mzek$$ which induces a morphism $$F_{m+k}\git G\to \Mzek.$$ By
Lemma \ref{lem-PicNumMze}, the Picard number of $\Mzek$ coincides
with that of $F_{m+k}\git G$ given in \eqref{eq4-13}. Hence the
morphism $F_{m+k}\git G\to \Mzek$ is an isomorphism as desired.
This completes our proof of Theorem \ref{thm4-1}.

\begin{remark}\label{2010rem1}
Let $S \subset \{1, 2, \cdots, n\}$ with $|S| = m-k$. On
$\overline{M}_{0,n\cdot \epsilon_{k}}$, the blow-up center for
$\overline{M}_{0,n\cdot \epsilon_{k+1}}\to\Mzek$ is the union of
$n \choose {m-k}$ smooth subvarieties $\Sigma_{n-m+k}^S\git G$.
Each $\Sigma_{n-m+k}^S \git G$ parameterizes weighted pointed
stable curves with $m-k$ colliding marked points $s_{i_1},
s_{i_2}, \cdots, s_{i_{m-k}}$ for $i_j \in S$. On the other hand,
for any member of $\Mzek$, no $m-k+1$ marked points can collide.
So we can replace $m-k$ marked points $s_{i_j}$ with $i_j \in S$
by a single marked point which cannot collide with any other
marked points. Therefore, an irreducible component
$\Sigma_{n-m+k}^S \git G$ of the blow-up center is isomorphic to
the moduli space of weighted pointed rational curves
$\overline{M}_{0, (1, \epsilon_k, \cdots, \epsilon_k)}$ with
$n-m+k+1$ marked points as discovered by Hassett. (See Proposition
\ref{reduction}.)
\end{remark}

\begin{remark}
For the moduli space of \emph{unordered} weighted pointed stable
curves $\Mzek/S_n$, we can simply take quotients by the $S_n$
action of the blow-up process in Theorem \ref{thm4-1}. In
particular, $\Mzn/S_n$ is obtained by a sequence of weighted
blow-ups from $\left((\PP^1)^n\git G\right)/ S_n=\PP^n\git G.$
\end{remark}


\section{Log canonical models of $\Mzn$}\label{sec6}

In this section, we give a relatively elementary and
straightforward proof of the following theorem of M. Simpson by
using Theorem \ref{thm4-1}. Let $M_{0,n}$ be the moduli space of
$n$ \emph{distinct} points in $\PP^1$ up to $\mathrm{Aut}(\PP^1)$.

\begin{theorem} \label{thm6.1}\emph{(M. Simpson \cite{Simpson})}
Let $\alpha$ be a rational number satisfying $\frac{2}{n-1} <
\alpha \le 1$ and let $D=\Mzn-M_{0,n}$ denote the boundary
divisor. Then the log canonical model
\[
    \Mzn(\alpha) =
    \proj\left(\bigoplus_{l \ge 0} H^0(\Mzn, \cO(\lfloor l (K_{\Mzn}
    +\alpha D) \rfloor))\right)
\]
satisfies the following:
\begin{enumerate}
\item If $\frac{2}{m-k+2} < \alpha \le \frac{2}{m-k+1}$ for $1\le k\le
m-2$, then $\Mzn(\alpha) \cong \Mzek$.\item  If $\frac{2}{n-1} <
\alpha \le \frac{2}{m+1}$, then $\Mzn(\alpha) \cong (\PP^1)^n \git
G$ where the quotient is taken with respect to the symmetric
linearization $\cO(1,\cdots,1)$.
\end{enumerate}
\end{theorem}
\begin{remark} Keel and McKernan prove (\cite[Lemma 3.6]{KeelMcKer}) that
$K_{\Mzn} + D$ is ample. Because $$\Mzemt\cong\Mzemo= \Mzn$$ by
definition, we find that (1) above holds for $k=m-1$ as
well.\end{remark}

For notational convenience, we denote $(\PP^1)^n\git G$ by $\Mzez$
for even $n$ as well. Let $\Sigma_k^S$ denote the subvarieties of
$F_k$ defined in \S\ref{sec3} for $S\subset \{1,\cdots,n\}$,
$|S|\le m$. Let
$$D_k^S=\Sigma_{n-m+k}^S\git G\subset F_{n-m+k}\git G\cong \Mzek.$$
Then $D_k^S$ is a divisor of $\Mzek$ for $|S|=2$ or $m-k<|S|\le
m$. Let $D_k^j=(\cup_{|S|=j}\Sigma_{n-m+k}^S)\git G$ and
$D_k=D_k^2+\sum_{j>m-k}D^j_k$. Then $D_k$ is the boundary divisor
of $\Mzek$, i.e. $\Mzek-M_{0,n}=D_k$. When $k = m-2$ so $\Mzek \cong
\Mzn$, sometimes we will drop the subscript $k$. Note that if $n$ is even and
$|S|=m$, $D_k^S=D_k^{S^c}=\Sigma^{S,S^c}_{n-m+k}\git G$.

By Theorem \ref{thm4-1}, there is a sequence of blow-ups
\begin{equation}\label{eq6.1}
\Mzn\cong \Mzemt\mapright{\varphi_{m-2}}
\Mzemth\mapright{\varphi_{m-3}}\cdots
\mapright{\varphi_{2}}\Mzeo\mapright{\varphi_{1}} \Mzez
\end{equation}
whose centers are transversal unions of smooth subvarieties,
except for $\varphi_1$ when $n$ is even. Note that the irreducible
components of the blow-up center of $\varphi_k$ furthermore
intersect transversely with $D^j_{k-1}$ for $j>m-k+1$ by Lemma
\ref{lem3-1} and by taking quotients.
\begin{lemma}\label{computepushandpull} Let $1\le k\le m-2$.
\begin{enumerate}
\item $\varphi_k^* (D^j_{k-1}) = D^j_k$ for $j> m-k+1$.
\item $\varphi_k^* (D^2_{k-1}) = D^2_k + {m-k+1 \choose 2} D^{m-k+1}_k$.
\item $\varphi_{k *} (D^j_k) = D^j_{k-1}$ for $j > m-k+1$ or $j=2$.
\item $\varphi_{k *} (D^j_k) =0$ for $j=m-k+1$.\end{enumerate}
\end{lemma}

\begin{proof} The push-forward formulas (3) and (4) are obvious.
Recall from \S\ref{sec4} that $\varphi_k=\bar\psi_{n-m+k}$ is the
quotient of $\psi_{n-m+k}:F_{n-m+k}^{ss}\to F_{n-m+k-1}^{ss}$.
Suppose $n$ is not even or $k$ is not $1$. Since $D^S_k$ for
$|S|>2$ does not contain any component of the blow-up center,
$\varphi_k^*(D^S_{k-1}) = D^S_k$. If $|S|=2$, $D^S_{k-1}$ contains
a component $D^{S'}_{k-1}$ of the blow-up center if and only if
$S' \supset S$. Therefore we have
\[
    \varphi_k^*(D^S_{k-1}) = D^S_k +
    \sum_{S' \supset S, |S'| = m-k+1} D^{S'}_k.
\]
By adding them up for all $S$ such that $|S|=2$, we obtain (2).

When $n$ is even and $k=1$, we calculate the pull-back before
quotient. Let $\pi:\tilde{F}_{m}^s\to F_m^{ss}$ be the map
obtained by blowing up $\cup_{|S|=m}\Sigma^{S,S^c}_m$ and removing
unstable points. Recall that $\tilde{F}_m^s/G\cong F_{m+1}\git
G\cong \Mzeo$ and the quotient of $\pi$ is $\varphi_1$. Then a
direct calculation similar to the above gives us
$\pi^*\Sigma^2_m=\tilde{\Sigma}_m^2+2\binom{m}{2}\tilde{\Sigma}^m_m$
where $\Sigma^2_m=\cup_{|S|=2}\Sigma^S_m$ and $\tilde{\Sigma}_m^2$
is the proper transform of $\Sigma^2_m$ while $\tilde{\Sigma}_m^m$
denotes the exceptional divisor. Note that by the descent lemma
(\cite{DN}), the divisor $\Sigma_m^2$ and $\tilde{\Sigma}^2_m$
descend to $D^2_0$ and $D_{1}^2$. However $\tilde{\Sigma}_m^m$
does not descend because the stabilizer group $\ZZ_2$ in $\bar
G=PGL(2)$ of points in $\tilde{\Sigma}_m^m$ acts nontrivially on
the normal spaces. But by the descent lemma again,
$2\tilde{\Sigma}^m_m$ descends to $D^m_1$. Thus we obtain (2).
\end{proof}

Next we calculate the canonical divisors of $\Mzek$.
\begin{proposition}\label{canonicaldiv1} \cite[Proposition
1]{Pand} The canonical divisor of $\Mzn$ is
\[
    K_{\Mzn} \cong -\frac{2}{n-1} D^2 + \sum_{j=3}^m
    \left(-\frac{2}{n-1}{j \choose 2} +(j-2)\right) D^j.
\]
\end{proposition}

\begin{lemma}\label{canonicaldiv2}
(1) The canonical divisor of $(\PP^1)^n \git G$ is
\[
    K_{(\PP^1)^n \git G} \cong -\frac{2}{n-1}D_0^2.
\]
(2) For $1\le k\le m-2$, the canonical divisor of $\Mzek$ is
\[
    K_{\Mzek} \cong -\frac{2}{n-1}D_k^2 + \sum_{j\ge m-k+1}^m
    \left(-\frac{2}{n-1}{j \choose 2} + (j-2) \right)D_k^j.
\]
\end{lemma}

\begin{proof}
It is well known by the descent lemma (\cite{DN}) that
$\mathrm{Pic}((\PP^1)^n \git G)$ is a free abelian group of rank
$n$(See \S6).  The symmetric group $S_n$ acts on $(\PP^1)^n \git G$ in the
obvious manner, and there is an induced action on its Picard
group. Certainly the canonical bundle $K_{(\PP^1)^n \git G}$ and
$D^2_0$ are $S_n$-invariant. On the other hand, the
$S_n$-invariant part of the rational Picard group is a one
dimensional vector space generated by the quotient $D_0^2$ of
$\cO_{(\PP^1)^n}(n-1,\cdots,n-1)$ and hence we have $K_{(\PP^1)^n
\git G} \cong c D_0^2$ for some $c\in \QQ$.

Suppose $n$ is odd. The contraction morphisms $\varphi_k$ are all
compositions of smooth blow-ups for $k \ge 1$. From the blow-up
formula of canonical divisors (\cite[II Exe. 8.5]{Hartshorne}) and
Lemma \ref{computepushandpull}, we deduce that
\[
    K_{\Mzek} = cD^2_k + \sum_{j\ge m-k+1}^m
    \left(c{j \choose 2} + (j-2)\right)D^j_k.
\]
Since $\Mzn \cong \Mzemt$, we get $c = -\frac{2}{n-1}$ from
Proposition \ref{canonicaldiv1}.

When $n$ is even, $\varphi_1^*(K_{(\PP^1)^n \git G}) = cD_1^2 +
c{m \choose 2} D^m_1$ by Lemma \ref{computepushandpull}. We write
$K_{\Mzeo} = cD_1^2 + (c{m \choose 2} + a)D_1^m$. By the blow-up
formula of canonical divisors (\cite[II Exe. 8.5]{Hartshorne})
again, we deduce that
\[
    K_{\Mzek} = cD_k^2 + \sum_{j \ge m-k+1}^{m-1}
    \left( c{j \choose 2} + (j-2)\right)D_k^j + (c{m \choose 2} + a)D_k^m.
\]
From Proposition \ref{canonicaldiv1} again, we get $c =
-\frac{2}{n-1}$ and $a = m-2$.
\end{proof}

We are now ready to prove Theorem \ref{thm6.1}. By \cite[Corollary
3.5]{Simpson}, the theorem is a direct consequence of the
following proposition.
\begin{proposition}\label{prop-amplerange}
(1) $K_{\Mzez}+\alpha D_0$ is ample if $\frac{2}{n-1} < \alpha \le
\frac{2}{m+1}$.\\
(2) For $1\le k\le m-2$, $K_{\Mzek}+\alpha D_k$ is ample if
$\frac{2}{m-k+2} < \alpha \le \frac{2}{m-k+1}$.

\end{proposition}

Since any positive linear combination of an ample divisor and a
nef divisor is ample \cite[Corollary 1.4.10]{Larz1}, it suffices
to show the following:
\begin{enumerate}\item[(a)] Nefness of $K_{\Mzek}+\alpha D_k$ for
$\alpha =\frac{2}{m-k+1}+s$ where $s$ is some (small) positive
number;
\item[(b)] Ampleness of $K_{\Mzek}+\alpha D_k$ for
$\alpha=\frac{2}{m-k+2}+t$ where $t$ is \emph{any} sufficiently
small positive number. \end{enumerate} We will use Alexeev and
Swinarski's intersection number calculation in \cite{AlexSwin} to
achieve (a) (See Lemma \ref{lem-otherextreme}.) and then (b) will
immediately follow from our Theorem \ref{thm4-1}.

\begin{definition} (\cite{Simpson}) Let $\varphi=\varphi_{n\cdot\epsilon_{m-2},
n\cdot\epsilon_k} : \Mzn \to \Mzek$ be the natural contraction map
(\S\ref{sec2.1}). For $k = 0,1,\cdots,m-2$ and $\alpha > 0$,
define $A(k, \alpha)$ by
\begin{eqnarray*}
    A(k, \alpha) &:=& \varphi^*(K_{\Mzek}+\alpha D_k)\\
    &=& \sum_{j = 2}^{m-k} {j \choose 2}\left(\alpha - \frac{2}{n-1}\right)D^j
    + \sum_{j \ge m-k+1}^m \left(\alpha - \frac{2}{n-1}{j \choose 2}
    + j-2\right)D^j.
\end{eqnarray*}
\end{definition}
Notice that the last equality is an easy consequence of Lemma
\ref{computepushandpull}.

By \cite{Kapranov}, there is a birational morphism $\pi_{\vec{x}}
: \Mzn \to (\PP^1)^n \git_{\vec{x}} G$ for any linearization
$\vec{x} = (x_1, \cdots, x_n) \in \QQ_+^n$. Since the canonical
ample line bundle $\cO_{(\PP^1)^n }(x_1, \cdots, x_n)\git G$ over
$(\PP^1)^n \git_{\vec{x}} G$ is ample, its pull-back $L_{\vec{x}}$
by $\pi_{\vec{x}}$  is certainly nef.

\begin{definition}\cite[Definition 2.3]{AlexSwin}\label{def-symnefdiv}
Let $x$ be a rational number such that $\frac{1}{n-1} \le x \le \frac{2}{n}$.
Set $\vec{x} = \cO(x, \cdots, x, 2-(n-1)x)$.
Define
\[
    V(x, n) := \frac{1}{(n-1)!}\bigotimes_{\tau \in S_n} L_{\tau
    \vec{x}}.
\]
Obviously the symmetric group $S_n$ acts on $\vec{x}$ by permuting
the components of $\vec{x}$.
\end{definition}
Notice that $V(x,n)$ is nef because it is a positive linear
combination of nef line bundles.

\begin{definition}\cite[Definition 3.5]{AlexSwin}
Let $C_{a,b,c,d}$ be \emph{any} vital curve class corresponding to
a partition $S_a \sqcup S_b \sqcup S_c \sqcup S_d$ of $\{1, 2,
\cdots, n\}$
such that $|S_a| = a, \cdots, |S_d|=d$.\\
(1) Suppose $n=2m+1$ is odd. Let $C_i = C_{1, 1, m-i, m+i-1}$,
for $i = 1, 2, \cdots, m-1$.\\
(2) Suppose $n=2m$ is even. Let $C_i = C_{1,1,m-i, m+i-2}$ for $i
= 1, 2, \cdots, m-1$.
\end{definition}

By \cite[Corollary 4.4]{KeelMcKer}, the following computation is straightforward.
\begin{lemma}\label{lem-intAkalpha}
The intersection numbers $C_i \cdot A(k, \alpha)$ are
\[
    C_i \cdot A(k, \alpha) = \left\{\begin{array}{ll}
    \alpha&\mbox{if } i < k\\
    \left(2-{m-k \choose 2}\right)\alpha + m-k-2&\mbox{if } i = k\\
    \left({m-k+1 \choose 2}-1\right)\alpha-m+k+1
    &\mbox{if } i = k+1\\
    0&\mbox{if } i > k+1.
    \end{array}\right.
\]
\end{lemma}

This lemma is in fact a slight generalization of \cite[Lemma
3.7]{AlexSwin} where the intersection numbers for
$\alpha=\frac{2}{m-k+1}$ only are calculated.

The $S_n$-invariant subspace of Neron-Severi vector space of
$\Mzn$ is generated by $D^j$ for $j=2,3, \cdots, m$ (\cite[Theorem
1.3]{KeelMcKer}). Therefore, in order to determine the linear
dependency of $S_n$-invariant divisors, we find $m-1$ linearly
independent curve classes, and calculate the intersection numbers
of divisors with these curves classes. Let $U$ be an $(m-1) \times
(m-1)$ matrix with entries $U_{ij} = (C_i \cdot V(\frac{1}{m+j},
n))$ for $1 \le i,j \le m-1$. Since $V(\frac{1}{m+j}, n)$'s are
all nef, all entries of $U$ are nonnegative.
\begin{lemma}\cite[\S3.2, \S3.3]{AlexSwin}\label{lem-nefbdl}
(1) The intersection matrix $U$ is upper triangular and
if $i \le j$, then $U_{ij} > 0$. In particular, $U$ is invertible.\\
(2) Let $\vec{a} = ((C_1 \cdot A(k, \frac{2}{m-k+1})), \cdots,
(C_{m-1} \cdot A(k,\frac{2}{m-k+1})))^t$ be the column vector
of intersection numbers.
Let $\vec{c} = (c_1, c_2, \cdots, c_{m-1})^t$ be the
unique solution of the system of linear equations
$U \vec{c} = \vec{a}$.
Then $c_i > 0$ for $i \le k+1$ and $c_i = 0$ for $i \ge k+2$.
\end{lemma}

This lemma implies that $A(k,\frac{2}{m-k+1})$ is a positive
linear combination of $V(\frac{1}{m+j}, n)$ for $j = 1, 2, \cdots,
k+1$. Note that $A(k,\frac{2}{m-k+2}) =
A(k-1,\frac{2}{m-(k-1)+1})$ and that for $\frac{2}{m-k+2} \le
\alpha \le \frac{2}{m-k+1}$, $A(k,\alpha)$ is a nonnegative linear
combination of $A(k,\frac{2}{m-k+2})$ and $A(k,\frac{2}{m-k+1})$.
Hence by the numerical result in Lemma \ref{lem-nefbdl} and the
convexity of the nef cone, $A(k,\alpha)$ is nef for
$\frac{2}{m-k+2} \le \alpha \le \frac{2}{m-k+1}$. Actually we can
slightly improve this result by using continuity.

\begin{lemma}\label{lem-otherextreme}
For each $k = 0,1,\cdots,m-2$, there exists $s > 0$ such that
$A(k,\alpha)$ is nef for $\frac{2}{m-k+2} \le \alpha \le
\frac{2}{m-k+1}+s$. Therefore, $K_{\Mzek}+\alpha D_k$ is nef for
$\frac{2}{m-k+2} \le \alpha \le \frac{2}{m-k+1}+s$.
\end{lemma}
\begin{proof}
Let $\vec{a}^\alpha = ((C_1 \cdot A(k,\alpha)), \cdots, (C_{m-1}
\cdot A(k,\alpha)))^t$ and let $\vec{c}^\alpha = (c^\alpha_1,
\cdots, c^\alpha_{m-1})^t$ be the unique solution of equation $U
\vec{c}^\alpha = \vec{a}^\alpha$. Then by continuity, the
components $c^\alpha_1, c^\alpha_2, \cdots, c^\alpha_{k+1}$ remain
positive when $\alpha$ is slightly increased. By Lemma
\ref{lem-intAkalpha} and the upper triangularity of $U$,
$c^\alpha_i$ for $i > k+1$ are all zero. Hence $A(k, \alpha)$ is
still nef for $\alpha = \frac{2}{m-k+1}+s$ with sufficiently small
$s > 0$.
\end{proof}

With this nefness result, the proof of Proposition
\ref{prop-amplerange} is obtained as a quick application of
Theorem \ref{thm4-1}.

\begin{proof}[Proof of Proposition \ref{prop-amplerange}]
We prove that in fact $K_{\Mzek}+\alpha D_k$ is ample for
$\frac{2}{m-k+2} < \alpha < \frac{2}{m-k+1}+s$ where $s$ is the
small positive rational number in Lemma \ref{lem-otherextreme}.
Since a positive linear combination of an ample divisor and a nef
divisor is ample by \cite[Corollary 1.4.10]{Larz1}, it suffices to
show that $K_{\Mzek}+\alpha D_k$ is ample when $\alpha=
\frac{2}{m-k+2}+t$ for any sufficiently small $t>0$ by Lemma
\ref{lem-otherextreme}.

We use induction on $k$. It is certainly true when $k=0$ by Lemma
\ref{canonicaldiv2} because $D^2_0$ is ample as the quotient of
$\cO(n-1,\cdots,n-1)$. Suppose $K_{\Mzeko}+\alpha D_{k-1}$ is ample
for $\frac{2}{m-k+3} < \alpha < \frac{2}{m-k+2}+s'$ where
$s'$ is the small positive number in Lemma \ref{lem-otherextreme}
for $k-1$.
Since $\varphi_k$ is a blow-up with exceptional divisor $D_k^{m-k+1}$,
\[
    \varphi_k^*(K_{\Mzeko}+\alpha D_{k-1})-\delta D_k^{m-k+1}
\]
is ample for any sufficiently small $\delta>0$ by \cite[II
7.10]{Hartshorne}. A direct computation with Lemmas
\ref{computepushandpull} and \ref{canonicaldiv2} provides us with
\[
\varphi_k^*(K_{\Mzeko}+\alpha D_{k-1})-\delta D_k^{m-k+1}\]  \[
=K_{\Mzek}+\alpha D_{k}+\left( \binom{m-k+1}{2}\alpha
-\alpha-(m-k-1)-\delta \right)D_k^{m-k+1}.
\]
If $\alpha=\frac{2}{m-k+2}$, $\binom{m-k+1}{2}\alpha
-\alpha-(m-k-1)=0$ and thus we can find $\alpha>\frac{2}{m-k+2}$
satisfying $\binom{m-k+1}{2}\alpha -\alpha-(m-k-1)-\delta  =0$. If
$\delta$ decreases to $0$, the solution $\alpha$ decreases to
$\frac{2}{m-k+2}$. Hence $K_{\Mzek}+\alpha D_k$ is ample when
$\alpha= \frac{2}{m-k+2}+t$ for any sufficiently small $t>0$ as
desired.
\end{proof}

\begin{remark}\label{rem-compproofs}
There are already two different proofs of M. Simpson's theorem
(Theorem \ref{thm6.1}) given by Fedorchuk--Smyth \cite{FedoSmyt},
and by Alexeev--Swinarski \cite{AlexSwin} without relying on
Fulton's conjecture. Here we give a brief outline of the two
proofs.

In \cite[Corollary 3.5]{Simpson}, Simpson proves that Theorem
\ref{thm6.1} is an immediate consequence of the ampleness of
$K_{\Mzek} + \alpha D_k$ for $\frac{2}{m-k+2} < \alpha \le
\frac{2}{m-k+1}$ (Proposition \ref{prop-amplerange}). The differences in
the proofs of Theorem \ref{thm6.1} reside solely in different ways
of proving Proposition \ref{prop-amplerange}.

The ampleness of $K_{\Mzek} + \alpha D_k$ follows if the divisor
$A(k, \alpha) = \varphi^*(K_{\Mzek}+\alpha D_k)$ is nef and its
linear system contracts only $\varphi$-exceptional curves. Here,
$\varphi : \Mzn \to \Mzek$ is the natural contraction map
(\S\ref{sec2.1}). Alexeev and Swinarski prove Proposition
\ref{prop-amplerange} in two stages: First the nefness of $A(k,
\alpha)$ for suitable ranges is proved and next they show that the
divisors are the pull-backs of ample line bundles on $\Mzek$.
Lemma \ref{lem-otherextreme} above is only a negligible
improvement of the nefness result in \cite[\S3]{AlexSwin}. In
\cite[Theorem 4.1]{AlexSwin}, they give a criterion for a line
bundle to be the pull-back of an ample line bundle on $\Mzek$.
After some rather sophisticated combinatorial computations, they
prove in \cite[Proposition 4.2]{AlexSwin} that $A(k, \alpha)$
satisfies the desired properties.

On the other hand, Fedorchuk and Smyth show that $K_{\Mzek} +
\alpha D_k$ is ample as follows. Firstly, by applying the
Grothendieck-Riemann-Roch theorem, they represent
$K_{\Mzek}+\alpha D_k$ as a linear combination of boundary
divisors and tautological $\psi$-classes. Secondly, for such a
linear combination of divisor classes and for a complete curve in
$\Mzek$ parameterizing a family of curves with smooth general
member, they perform brilliant computations and get several
inequalities satisfied by their intersection numbers
(\cite[Proposition 3.2]{FedoSmyt}). Combining these inequalities,
they prove in particular that $K_{\Mzek}+\alpha D_k$ has positive
intersection with any complete curve on $\Mzek$ with smooth
general member (\cite[Theorem 4.3]{FedoSmyt}). Thirdly, they prove
that if the divisor class intersects positively with any curve
with smooth general member, then it intersects positively with all
curves by an induction argument on the dimension. Thus they
establish the fact that $K_{\Mzek}+\alpha D_k$ has positive
intersection with all curves. Lastly, they prove that the same
property holds even if $K_{\Mzek}+\alpha D_k$ is perturbed by any
small linear combination of boundary divisors. Since the boundary
divisors generate the Neron-Severi vector space, $K_{\Mzek}+\alpha
D_k$ lies in the interior of the nef cone and the desired
ampleness follows.
\end{remark}

\section{The Picard groups of $\Mzek$}\label{sec5}

As a byproduct of our GIT construction of the moduli spaces of
weighted pointed curves, we give a basis of the \emph{integral}
Picard group of $\Mzek$ for $0 \le k \le m-2$.

Let $e_i$ be the $i$-th standard basis vector of $\ZZ^n$. For
notational convenience, set $e_{n+1} = e_1$. For $S\subset
\{1,2,\cdots,n\}$, let $D_k^S=\Sigma_{n-m+k}^S\git G\subset
F_{n-m+k}\git G\cong \Mzek$. Note that if $m-k<|S|\le m$ or
$|S|=2$, $D_k^S$ is a divisor of $\Mzek$.

\begin{theorem}\label{picardgroup}
(1) If $n$ is odd, then the Picard group of $\Mzek$ is
\[
    \mathrm{Pic} (\Mzek) \cong \bigoplus_{m-k < |S| \le m} \ZZ D^S_k \oplus
    \bigoplus_{i=1}^n \ZZ D_k^{\{i, i+1\}}
\]
for $0 \le k \le m-2$.\\
(2) If $n$ is even, then the Picard group of $\Mzek$ is
\[
    \mathrm{Pic} (\Mzek) \cong \bigoplus_{m-k < |S| < m} \ZZ D^S_k \oplus
    \bigoplus_{1 \in S, |S| = m} \ZZ D^S_k \oplus
    \bigoplus_{i=1}^{n-1} \ZZ D_k^{\{i, i+1\}}
    \oplus \ZZ D_k^{\{1, n-1\}}.
\]
for $1 \le k \le m-2$.
\end{theorem}

\begin{proof}
Since the codimensions of unstable strata in $(\PP^1)^n$ are greater
than 1, $$\mathrm{Pic} (((\PP^1)^n)^{ss}) = \mathrm{Pic} ((\PP^1)^n)
\cong \oplus_{1 \le i \le n} \ZZ \cO(e_i).$$ For all $x \in
((\PP^1)^n)^s$, $G_x \cong \{\pm 1\}$. If $n$ is even and $x$ is
strictly semistable point with closed orbit, then $G_x \cong \CC^*$.
Since $G$ is connected, $G$ acts on the discrete set
$\mathrm{Pic}((\PP^1)^n)$ trivially. By Kempf's descent lemma
(\cite[theorem 2.3]{DN}) and by checking the actions of the
stabilizers on the fibers of line bundles, we deduce that $\cO(a_1,
a_2, \cdots, a_n)$ descends to $((\PP^1)^n)^{ss} \git G$ if and only
if $2$ divides $\sum a_i$.

Consider the case when $n$ is odd first. It is elementary to check
that the subgroup $\{(a_1, \cdots, a_n) \in \ZZ^n | \sum a_i \in
2\ZZ \}$ is free abelian of rank $n$ and $\{e_i + e_{i+1}\}$ for
$1 \le i \le n$ form a basis of this group. Furthermore, for
$S=\{i,j\}$ with $i\ne j$, the big diagonal $(\Sigma_{m+1}^S)^s =
(\Sigma_0^S)^s$ satisfies $\cO((\Sigma_0^S)^s) \cong
\cO_{F_0^s}(e_i+e_j)$. Hence in $F_{m+1}\git G = F_0 \git G$,
$\cO(\Sigma_0^S \git G) \cong \cO(e_i + e_j)$. Therefore we have
\[
    \mathrm{Pic} (\Mzez) = \mathrm{Pic} (F_{m+1} \git G)
    = \bigoplus_{i = 1}^n \ZZ D_0^{\{i, i+1\}}.
\]

By Theorem \ref{thm4-1}, the contraction morphism $\varphi_k :
\Mzek \to \Mzeko$ is the blow-up along the union of transversally
intersecting smooth subvarieties. By \S2.3, this is a composition
of smooth blow-ups. In $\Mzek$, the exceptional divisors are
$D^S_k$ for $|S| = {m-k+1}$. So the Picard group of $\Mzek$ is
\[
    \mathrm{Pic} (\Mzek) \cong \varphi_k^* \mathrm{Pic} (\Mzeko) \oplus
    \bigoplus_{|S| = {m-k+1}}\ZZ D^S_k.
\]
by \cite[II Exe. 8.5]{Hartshorne}. For any $S$ with $|S| = 2$,
$D_{k-1}^S$ contains the blow-up center $D_{k-1}^{S'}$ if $S
\subset {S'}$. So $\varphi_k^* (D_{k-1}^S)$ is the sum of $D_k^S$
and a linear combination of $D_k^{S'}$ for $S' \supset S, |S'| =
m-k+1$. If $|S| > 2$, then $\varphi_k^* D^S_{k-1} = D^S_k$ since
it does not contain any blow-up centers. After obvious basis
change, we get the desired result by induction.

Now suppose that $n$ is even. Still the group $\{ (a_1, \cdots,
a_n) \in \ZZ^n | \sum a_i \in 2\ZZ\}$ is free abelian of rank $n$
and $\{e_i + e_{i+1}\}_{1 \le i \le n-1} \cup \{e_1 + e_{n-1}\}$
form a basis. In $F_m\git G = F_0 \git G$, $\cO(\Sigma_m^S\git G)
\cong \cO(e_i + e_j)$ when $S = \{i, j\}$ with $i\ne j$. Hence
\[
    \mathrm{Pic} (F_m \git G) = \bigoplus_{i=1}^{n-1}
    \ZZ D_0^{\{i, i+1\}} \oplus \ZZ D_0^{\{1, n-1\}}.
\]

In $\tilde{F}_m$, the unstable loci have codimension two.
Therefore we have
\[
    \mathrm{Pic} (\tilde{F}_m^s) = \mathrm{Pic}(\tilde{F}_m)
    = \pi_m^* \mathrm{Pic} (F_m^{ss}) \oplus \bigoplus_{1 \in S, |S| = m}
    \ZZ \tilde{\Sigma}_m^S,
\]
where $\pi_m : \tilde{F}_m \to F_m^{ss}$ is the blow-up morphism,
and $\tilde{\Sigma}_m^S = \pi_m^{-1}(\Sigma_m^S \cap \Sigma_m^{S^c})$
for $|S|=m$.

By Kempf's descent lemma, $\mathrm{Pic}(F_{m+k}\git G)$ is a
subgroup of $\mathrm{Pic}(F_{m+k}^{ss})$ and
$\mathrm{Pic}(\tilde{F}_{m+k}^s)$ for $0 \le k \le m-2$. From our
blow-up description, all arrows except possibly
$\bar{\psi}_{m+1}^*$ in following commutative diagram
\[
    \xymatrix{\mathrm{Pic}(\tilde{F}_{m+1}^s)&
    \mathrm{Pic}(\tilde{F}_m^s)\ar@{=}[l]_{\tilde{\psi}_{m+1}^*}\\
    \mathrm{Pic}(F_{m+1}^{ss})\ar[u]_{\pi_{m+1}^*}&
    \mathrm{Pic}(F_m^{ss})\ar[l]_{\psi_{m+1}^*}\ar[u]_{\pi_m^*}\\
    \mathrm{Pic}(F_{m+1}\git G)\ar[u]&
    \mathrm{Pic}(F_m \git G)\ar[l]_{\bar{\psi}_{m+1}^*}\ar[u]}
\] are
injective, and thus the bottom arrow $\bar{\psi}_{m+1}^*$ is also
injective. Hence $\mathrm{Pic}(F_{m+1}\git G)$ contains the
pull-back of $\mathrm{Pic}(F_m \git G)$ as a subgroup. Also, for
the quotient map $p : \tilde{F}_{m+1}^s \to F_{m+1}\git G$, $p^*
D_1^S = \tilde{\Sigma}_{m+1}^S$ for $|S| = m$. Let $H$ be the
subgroup of $\mathrm{Pic}(\tilde{F}_{m+1}^s)$ generated by the
images of $\bar{\psi}_{m+1}^* \mathrm{Pic}(F_m \git G)$ and the
divisors $D_1^S$ with $|S| = m$. By definition, the image of
$\mathrm{Pic}(F_{m+1}\git G)$ contains $H$. Now by checking the
action of stabilizers on the fibers of line bundles, it is easy to
see that no line bundles in  $\mathrm{Pic}(\tilde{F}_{m+1}^s)- H$
descend to $F_{m+1}\git G$. Hence we have
\begin{equation}\label{mzeo1}
    \mathrm{Pic}(\Mzeo) = \mathrm{Pic}(F_{m+1} \git G)
    = \bar{\psi}_{m+1}^* \mathrm{Pic}(F_m \git G) \oplus
    \bigoplus_{1 \in S, |S| = m}\ZZ D_1^S.
\end{equation}

For an $S$ with $|S|=2$, $\Sigma_m^S\git G$ contains the blow-up
center $\Sigma_m^{S'} \cap \Sigma_m^{{S'}^c}\git G$ if $S \subset
S'$ or $S \subset {S'}^c$. So $\bar{\psi}_{m+1}^* (D_0^S)$ is the
sum of $D_{1}^S$ and a linear combination of divisors $D_1^{S'}$
for $S' \supset S$ or ${S'}^c\supset S$ with $|S'| = m$. From this
and (\ref{mzeo1}), we get the following by an obvious basis
change:
\[
    \mathrm{Pic}(\Mzeo) = \bigoplus_{1 \in S, |S| = m}\ZZ D_1^S
    \oplus \bigoplus_{i=1}^{n-1} \ZZ D_1^{\{i,i+1\}}
    \oplus \ZZ D_1^{\{1,n-1\}}.
\]
The rest of the proof is identical to the odd $n$ case and so we
omit it.
\end{proof}


\bibliographystyle{amsplain}

\end{document}